\include{BoxedEPS}

\documentclass[12pt]{amsart}
\usepackage{amsmath,amsthm}

\chardef\bslash=`\\ 

\makeatletter
\def\verbatim{\interlinepenalty\@M \@verbatim
  \leftskip\@totalleftmargin\advance\leftskip2pc
  \frenchspacing\@vobeyspaces \@xverbatim}
\makeatother
\newtheorem{thm}{Theorem}[section]
\newtheorem{cor}[thm]{Corollary}
\newtheorem{lem}[thm]{Lemma}
\newtheorem{prop}[thm]{Proposition}

\newtheorem{rem}[thm]{Remark} 

\numberwithin{equation}{section}

\newcommand{\begeq}{\begin {equation}}

\newcommand{\bs}{\begin {split}}
\newcommand{\es}{\end {split}}
\newcommand{\bp}{\begin {prop}}
\newcommand{\ep}{\end {prop}}
\newcommand{\bt}{\begin {thm}}
\newcommand{\et}{\end {thm}}
\newcommand{\bc}{\begin {cor}}
\newcommand{\ec}{\end {cor}}
\newcommand{\bl}{\begin {lem}}
\newcommand{\el}{\end {lem}}
\newcommand{\bpf}{\begin {proof}}
\newcommand{\epf}{\end {proof}}
\newcommand{\bi}{\begin {itemize}}
\newcommand{\ei}{\end {itemize}}
\newcommand{\ben}{\begin {enumerate}}
\newcommand{\een}{\end {enumerate}}
\newcommand{\brem}{\begin {rem}}
\newcommand{\erem}{\end {rem}}

\newcommand{\ZZ}{{\mathbb Z}}
\newcommand{\TT}{{\mathbb T}} 
\newcommand{\RR}{{\mathbb R}}
\newcommand{\CC}{{\mathbb C}}

\newcommand{\Rd}{ {\Bbb R}^d}
\newcommand{\Zd} {{\Bbb Z}^d}

\begin{document}
\bibliographystyle{plain}

\title[Wiener's lemma for Infinite Matrices II] {Wiener's Lemma for Infinite Matrices II}

\author{ Qiyu Sun }

\address{
Department of Mathematics,  University of Central Florida,
Orlando, FL 32816, USA}
\email{qsun@mail.ucf.edu}


\date{\today }




\maketitle
\begin{abstract}
In this paper, we introduce a  class  of infinite matrices
related to  the Beurling algebra
of periodic functions, and we  show that
it is an  inverse-closed subalgebra of ${\mathcal B}(\ell^q_w)$, the algebra
 of all bounded linear operators on the weight sequence space  $\ell^q_w$, for any $1\le q<\infty$ and any discrete Muckenhoupt  $A_q$-weight $w$.
\end{abstract}


\section{Introduction}
Let us begin the sequel to \cite{suntams07} by introducing a new class of infinite matrices,
\begin{equation}\label{beurling.def_origininal}
{\mathcal B}(\Zd, \Zd):=\Big\{ A:=(a(i,j))_{i,j\in \Zd}\ \Big|\  \|A\|_{\mathcal B(\Zd, \Zd)}<\infty\Big\},
\end{equation}
where $d\ge 1$, $|x|_\infty:=\max(|x_1|, \ldots, |x_d|)$ for $x:=(x_1, \ldots, x_d)\in \Rd$, and
\begin{equation}\label{beurlingnorm.def_original}
\|A\|_{{\mathcal B}(\Zd, \Zd)}:=\sum_{k\in \Zd}\Big(\sup_{|i-j|_\infty\ge |k|_\infty} |a(i,j)|\Big).
\end{equation}

It is observed that a Toeplitz matrix $A:=(a(i-j))_{i,j\in \ZZ}$ associated with a sequence $a:=(a(n))_{n\in \ZZ}$
belongs to ${\mathcal B}(\ZZ, \ZZ)$ if and only if the Fourier series $\hat a(\xi):=\sum_{n\in \ZZ} a(n)\exp(-\sqrt{-1}\ n\xi)$
belongs to the  algebra
\begin{equation}\label{classicalbeurling.def}
A^*(\TT):=\Big\{\sum_{n=-\infty}^\infty a(n) e^{-\sqrt{-1} \ n\xi }\ \Big| \ \sum_{k=0}^\infty \sup_{|n|\ge k} |a(n)|<\infty\Big\}.
\end{equation}
The above algebra $A^*(\TT)$
 was introduced by A. Beurling for establishing contraction properties of periodic functions \cite{beurling49}, and was
 used  in considering pointwise summability of Fourier series
\cite{ bo65,fe85, fw06, steinbook93,  
te73}.
So   the class ${\mathcal B}(\Zd,  \Zd)$ of infinite matrices can be interpreted  as a
noncommutative matrix extension of the {\em Beurling algebra} $A^*(\TT)$.

Define
the  {\em Gr\"ochenig-Schur class}  ${\mathcal S}(\Zd, \Zd)$ by
\begin{eqnarray}\label{schurold.def}
{\mathcal S}(\Zd, \Zd)  & := & \Big\{ \big(a(i,j)\big)_{i,j\in \Zd} \ \Big | \nonumber\\
& &  \quad \Big( \sup_{i\in \Zd} \sum_{j\in \Zd} |a(i,j)|,
\sup_{j\in \Zd} \sum_{i\in \Zd} |a(i,j)|\Big)<\infty\Big\}
\end{eqnarray}
\cite{gltams06, schur11, suncasp05, suntams07},
and   the {\em
 Gohberg-Baskakov-Sj\"ostrand class} ${\mathcal C}(\Zd, \Zd)$
 by
 \begin{equation}\label{gbsold.def}
{\mathcal C}(\Zd, \Zd):=\Big\{ \big(a(i,j)\big)_{i,j\in \Zd}\ \Big| \
 \sum_{k\in \Zd} \big(\sup_{i-j=k} |a(i,j)|\big) <\infty\Big\}
\end{equation}
\cite{baskakov90, gkwieot89, gltams06, kurbatov90, sjostrand94,suntams07}.  The above two classes of infinite matrices appeared in the study of  Gabor
time-frequency analysis, nonuniform sampling, and algebra of
pseudo-differential operators etc (see \cite{akramgrochenig01,
balanchl04,  grochenigl03, grochenigstrohmer,   sjostrand94, sunsiam06} for a sample of papers).
   From \eqref{beurling.def_origininal}, \eqref{schurold.def} and
    \eqref{gbsold.def} it follows that
\begin{equation}\label{bcsinclusion.eq}
{\mathcal B}(\Zd, \Zd)\subset {\mathcal C}(\Zd, \Zd) \subset {\mathcal S}(\Zd, \Zd).\end{equation}
This shows that  any matrix in  ${\mathcal B}(\Zd, \Zd)$ belongs to
the Gr\"ochenig-Schur class
   ${\mathcal S}(\Zd, \Zd)$ and also
    the
 Gohberg-Baskakov-Sj\"ostrand class ${\mathcal C}(\Zd, \Zd)$.

An equivalent way of defining   ${\mathcal B}(\Zd, \Zd)$  is the existence of
a radially decreasing  sequence $\{b(i)\}_{i\in \Zd}$
for any infinite matrix  $A:=(a(i,j))_{i,j\in \Zd}\in {\mathcal B}(\Zd, \Zd)$ such that
 \begin{equation}
 |a(i,j)|\le b(i-j)\quad  {\rm for \ all} \ i,j\in \Zd,
 \end{equation}
  \begin{equation} \sum_{i\in \Zd} b(i)<\infty,
  \end{equation}
  and
  \begin{equation}
b(i)= h(|i|_\infty)
\ {\rm for \ some\   decreasing \ sequence} \ \{h(n)\}_{n=0}^\infty.
\end{equation}
Therefore any infinite matrix in
${\mathcal B}(\Zd, \Zd)$ is dominated by a convolution operator associated with
a summable, radial and (radially) decreasing sequence.
We remark that
any infinite matrix in the  Gohberg-Baskakov-Sj\"ostrand class ${\mathcal C}(\Zd, \Zd)$
 is dominated by a convolution operator associated with
a summable sequence \cite{baskakov90, gkwieot89, gltams06, kurbatov90, sjostrand94,suntams07}.

\bigskip

 A positive sequence $w:=(w(i))_{i\in \Zd}$ is said to be a
{\em discrete $A_q$-weight} for $1<q<\infty$ if there exists a positive constant $C$ such that
\begin{equation}\label{discreteaqweight.eq1}
\Big( N^{-d} \sum_{i\in a+[0,N-1]^d}  w(i)\Big)
\Big(N^{-d}\sum_{i\in a+[0,N-1]^d}  (w(i))^{-(q-1)} \Big)^{1/(q-1)}\le C
\end{equation}
hold for  all $a\in \Zd$  and    $1\le N\in \ZZ$, and  to be
a
{\em discrete $A_1$-weight}  if there exists a positive constant $C$ such that
\begin{equation}\label{discreteaqweight.eq2}
 N^{-d} \sum_{i\in a+[0,N-1]^d}  w(i)
\le C \inf_{i\in a+[0,N-1]^d}  w(i)
\end{equation}
hold for  all $a\in \Zd$  and    $1\le N\in \ZZ$ \cite{garciabook85, steinbook93}. The  smallest constant $C$ for which \eqref{discreteaqweight.eq1}  holds
when $1<q<\infty$,
and respectively for which \eqref{discreteaqweight.eq2} holds when $q=1$, to be denoted by $A_q(w)$, is the {\em discrete $A_q$-bound}.
The positive  sequences  $((1+|i|_\infty)^\alpha)_{i\in \Zd}$
 with $-d<\alpha<d(q-1)$ if $1<q<\infty$, and respectively with $-d<\alpha\le 0$ if $q=1$, are discrete $A_q$-weights.

For $1<q<\infty$, a positive  locally integrable  function $w$ on $\Rd$ is said to  be an
{\em $A_q$-weight}  if there exists a positive constant $C$ such that
\begin{equation}\label{apweight.eq1}
\Big(\frac{1}{|Q|} \int_Q w(x) dx\Big) \Big(\frac{1}{|Q|} \int_{Q} w(x)^{-(q-1)} dx\Big)^{1/(q-1)} \le C\end{equation}
for all cubic $Q\subset \Rd$, where $|Q|$ denotes
 the Lebesgue measure of the cubic $Q$ \cite{rosenblum62}. Similarly for $q=1$,
a  positive locally integrable  function $w$ is said to  be an
{\em $A_1$-weight}  if there exists a positive constant $C$ such that
\begin{equation}
\label{apweight.eq2}\frac{1}{|Q|} \int_Q w(y) dy \le C w(x),\quad \ x\in Q\end{equation}
for all cubic $Q\subset \Rd$ \cite{fs71}.  One may then verify that
a positive sequence $w:=\big(w(i)\big)_{i\in \Zd}$ is a discrete $A_q$-weight if and only if $\tilde w(x):=\sum_{i\in \Zd} w(i) \chi_{[-1/2, 1/2)^d} (x-i)$ is an $A_q$-weight, where $1\le q<\infty$ and $\chi_E$ is the characteristic function on a set $E$ \cite{lemarie94}.

For $1\le q<\infty$ and a positive sequence $w:=(w(i))_{i\in \Zd}$ on $\Zd$, let
$\ell^q_w:=\ell^q_w(\Zd)$ be
the  space of  all weighted $q$-summable sequences on $\Zd$, i.e.,
\begin{equation}
\ell^q_w(\Zd) := \Big \{ (c(i))_{i\in \Zd}\ \Big|\  \|c\|_{q,w}:=\Big(\sum_{i\in \Zd} |c(i)|^q w(i)\Big)^{1/q}<\infty\Big\}.
\end{equation}
For the trivial weight $w_0$ (i.e. $w_0(i)=1$ for all $i\in \Zd$), we will use
$\ell^q$ and $\|\cdot\|_{q}$ instead of $\ell^q_{w_0}$ and $\|\cdot\|_{q,w_0}$ for brevity.
Define the  {\em discrete maximal function} by
\begin{equation}
Mc(i):=\sup_{0\le N\in \ZZ} \ \frac{1}{(2N+1)^{d}} \sum_{k\in i+[-N, N]^d}
|c(k)| \quad {\rm for} \ c:=(c(i))_{i\in \Zd}.
\end{equation}
Similar to the characterization of  an $A_q$-weight on $\Rd$
via  the standard
maximal operator 
\cite{mtams72},
  the discrete $A_q$-weight can be characterized via the discrete maximal function on the weighted $\ell^q$ space. More precisely, a positive sequence $w:=(w(i))_{i\in \Zd}$ is a discrete $A_q$-weight if and only if the  discrete maximal operator $c\longmapsto Mc$ is of weak-type $(\ell^q_w, \ell^q_w)$, i.e., there exists a positive constant $C$ such that
\begin{equation}
\sum_{Mc(i)\ge\alpha} w(i)\le \frac{C}{\alpha^q} \|c\|_{q,w}^q \quad {\rm for\ all} \ \alpha>0\ {\rm and} \ c\in \ell^p_w.
\end{equation}
Moreover for $1<q<\infty$, the discrete maximal operator $M$ is of strong type $(\ell^q_w, \ell^q_w)$ for a discrete $A_q$-weight $w$, i.e.,
there exists a positive constant $C'$ such that
\begin{equation}
\|Mc\|_{q,w} \le C' \|c\|_{q,w} \quad {\rm for\ all} \ c\in \ell^p_w.
\end{equation}
The reader may refer to \cite{garciabook85} for a complete account of the theory of weighted inequalities and its ramification.

\bigskip

Now let us  present our results for infinite  matrices in
 ${\mathcal B}(\Zd, \Zd)$.
 In Section \ref{algebra.section}, we establish the following algebraic properties for the class ${\mathcal B}(\Zd, \Zd)$  of infinite matrices.

\begin{thm}\label{beurlingalgebra.tm} Let $1\le q<\infty$, and let  $w$ be a discrete  $A_q$-weight. Then
${\mathcal B}(\Zd, \Zd)$
is a unital Banach algebra under matrix multiplication, and also a subalgebra of ${\mathcal B}(\ell^q_w)$,  the algebra of all bounded linear  operators on the weight sequence space $\ell^q_w$.
\end{thm}

\bigskip

By Theorem \ref{beurlingalgebra.tm},
every infinite matrix $A\in {\mathcal B}(\Zd, \Zd)$
defines a bounded operator on $\ell^q_w$ for any $1\le q<\infty$ and for
 any  discrete $A_q$-weight $w$, i.e.,
 there exists a positive constant $C$ such that
\begin{equation}\label{upperboundnew.eq1}
\|Ac\|_{q, w}\le C \|c\|_{q, w} \quad {\rm for\ all} \ c\in \ell^q_w.
\end{equation}
Besides the above  boundedness of an infinite matrix
 on the weighted
sequence space $\ell^q_w$, it is natural to consider
 $\ell^q_w$-stability.
Here for  $1\le q<\infty$ and a positive sequence $w$ on $\Zd$,
we say that
 an infinite matrix $A$ has {\em $\ell^{q}_{w}$-stability} if
  there exists a positive constant $C$ such that
\begin{equation}
C^{-1} \|c\|_{q,w} \le \|Ac\|_{q,w} \le C\|c\|_{q,w}  \quad {\rm for \ all} \ c\in {\ell^q_w}.
\end{equation}
The $\ell^q_w$-stability  is one of the basic assumptions for infinite
matrices arising in the study of spline approximation,  Gabor
time-frequency analysis, nonuniform sampling, and algebra of
pseudo-differential operators etc (see \cite{akramgrochenig01,
balanchl04,  grochenigl03, grochenigstrohmer, jaffard90,   sjostrand94, sunsiam06,
suntams07}
 and the references therein.)
In Section \ref{stability.section}, we  establish the equivalence of
$\ell^q_w$-stabilities of  any infinite matrix in ${\mathcal B}(\Zd, \Zd)$ for different exponents $1\le q<\infty$ and for different  discrete
$A_q$-weights $w$.

\begin{thm}\label{beurlingstability.tm}
 Let $A\in {\mathcal B}(\Zd, \Zd)$. If $A$ has $\ell^q_w$-stability for some
$1\le q<\infty$ and for some discrete
$A_q$-weight $w$, then it
has $\ell^{q'}_{w'}$-stability for all
$1\le q'<\infty$ and for all discrete
$A_{q'}$-weights $w'$.
\end{thm}

The reader may refer to \cite{akramjfa09, shincjfa09, tessera09}
for the equivalence of unweighted  $\ell^q_w$-stability of
 infinite matrices in the Gohberg-Baskakov-Sj\"ostrand class ${\mathcal C}(\Zd, \Zd)$.
  If $A\in {\mathcal B}(\ell^q_w)$ has a {\em left inverse} $B\in {\mathcal B}(\ell^q_w)$, i.e., $BA=I$, then $A$ has $\ell^q_w$-stability. The converse is not true in general, unless $q=2$. As an application of  Theorem \ref{beurlingstability.tm}, we show that the converse holds for any infinite
 matrix $A$ in ${\mathcal B}(\Zd, \Zd)$.

 \begin{cor}\label{leftinverse.cor1}
  Let $1\le q< \infty$, and let $w$ be a discrete $A_q$-weight. Then an infinite matrix in ${\mathcal B}(\Zd, \Zd)$
   has $\ell^q_w$-stability if and only if it has a left inverse  in ${\mathcal B}(\ell^q_w)$.
 \end{cor}

\bigskip

 Given a Banach algebra ${\mathcal B}$,  a
subalgebra ${\mathcal A}$ of ${\mathcal B}$ is said to be {\em
inverse-closed}   if
$A\in {\mathcal A}$ and  the inverse $A^{-1}$ of the element $A$
exists in  ${\mathcal B}$  implies
that $A^{-1}\in {\mathcal A}$  \cite{gelfandbook, Naimarkbook, takesaki}.
The next question following the $\ell^q_w$-stability of an infinite matrix in ${\mathcal B}(\Zd, \Zd)$ is  whether its  inverse, if exists in ${\mathcal B}(\ell^q_w)$, belongs to ${\mathcal B}(\Zd, \Zd)$, or in the other word, whether ${\mathcal B}(\Zd, \Zd)$ is an inverse-closed subalgebra of ${\mathcal B}(\ell^q_w)$.

The inverse-closedness for the subalgebra of absolutely convergent Fourier series in the algebra of bounded periodic functions was first studied in
\cite{bochner, gelfandbook, Naimarkbook, wiener}.
The inverse-closed property (=Wiener's lemma)
has been established for infinite matrices satisfying
various off-diagonal decay conditions, see
 \cite{balan, balank10,
 baskakov90, fendler06, grochenigklotz10,
gltams06,  jaffard90, sjostrand94, suncasp05,
suntams07} for a sample of papers. Inverse-closedness occurs under various names (such as spectral invariance, Wiener pair, local subalgebra)
in many fields of mathematics, see the survey \cite{grochenigsurvey}.

The  inverse-closed property for non-commutative matrix subalgebra
 has been shown to be  crucial for  the
well-localization of dual wavelet frames
and  dual Gabor frames  \cite{balanchl04,  grochenigl03, jaffard90, krishtal10,  krishtaljat09},
the algebra of pseudo-differential operators \cite{grochenig06, grochenigstrohmer, sjostrand94},
the fast implementation in numerical analysis \cite{christensenstrohmer, dahlkejat09,  grochenigrsappear},  and the  local reconstruction in sampling theory \cite{akramgrochenig01, sunsiam06, sunaicm08}.

It mixes art and hard mathematical work to consider
 the  inverse-closed  subalgebra
  of ${\mathcal B}(\ell^q_w)$.  The art is to guess the off-diagonal decay of infinite matrices in an algebra ${\mathcal A}$, while the work is to prove the inverse-closedness of the algebra ${\mathcal A}$  in
  ${\mathcal B}(\ell^q_w)$.
There are several approaches to prove  the inverse-closedness
of a subalgebra of ${\mathcal B}(\ell^q_w)$.
 Here are three of them: (i) the indirect
approach, such as the Gelfand's technique
\cite{fendler06, gelfandbook,  gltams06}; (ii) the semi-direct approach, such as the
bootstrap argument \cite{jaffard90} and  the derivation trick \cite{grochenigklotz10}; (iii) the direct approach,
such as the commutator  trick \cite{baskakov90, sjostrand94} and
the asymptotic estimate technique \cite{shincjfa09, suncasp05, suntams07}.
 Each approach has its advantages and
limitations. For instance, the Gelfand technique and the asymptotic estimate technique
 work well for inverse-closed subalgebras of ${\mathcal B}(\ell^q)$ with $q=2$,
  but they are not directly applicable for  inverse-closed subalgebras of ${\mathcal B}(\ell^q)$ with $q\ne 2$.
The commutator trick is applicable to establish
 Wiener's lemma for subalgebra  of ${\mathcal B}(\ell^q), 1\le q\le \infty$
\cite{baskakov90, sjostrand94}.
In Section \ref{wiener.section}, we combine the commutator trick,  the
 asymptotic estimate technique and the  equivalence of $\ell^q_w$-stability for different exponents $q$ and for different discrete $A_q$-weights $w$,  and then  establish  Wiener's lemma for subalgebras of ${\mathcal B}(\ell^q_w)$, where $1\le q< \infty$ and  $w$ is a discrete $A_q$-weight.

\begin{thm}\label{wienerlemmaforbeurlingalgebra.tm}
Let $1\le  q<\infty$  and let
$w$ be a discrete $A_q$-weight. Then ${\mathcal B}(\Zd, \Zd)$ is an inverse-closed
subalgebra of ${\mathcal B}(\ell^q_w)$.
\end{thm}

As an application of Theorem \ref{wienerlemmaforbeurlingalgebra.tm}, we obain Wiener's lemma
 for the Beurling algebra $A^*(\TT)$ of periodic functions \cite{belinskiijfaa97}.

\begin{cor}\label{beurling.cor}
If $f\in A^*(\TT)$ and $f(\xi)\ne 0$ for all $\xi\in \RR$ then $ 1/f\in A^*({\TT})$.
\end{cor}

As applications of Theorems \ref{beurlingstability.tm}   and \ref{wienerlemmaforbeurlingalgebra.tm},  we  establish the equivalence between the $\ell^q_w$-stability of an infinite matrix in ${\mathcal B}(\Zd, \Zd)$ and the existence of its left inverse in ${\mathcal B}(\Zd, \Zd)$.

 \begin{cor}\label{leftinverse.cor2}
  Let $1\le q\le \infty$, and let $w$ be a discrete $A_q$-weight. Then an infinite matrix in ${\mathcal B}(\Zd, \Zd)$ has $\ell^q_w$-stability if and only if it has a left inverse  in ${\mathcal B}(\Zd, \Zd)$.
 \end{cor}

\section{A  class of infinite matrices}
\label{class.section}

In this section, we introduce a class of infinite matrices with off-diagonal decay,
which includes the class ${\mathcal B}(\Zd, \Zd)$ in the Introduction as a special case.

\bigskip

A {\em weight matrix on $\Zd\times \Zd$}, or a {\em weight matrix} for brevity, is  a positive matrix  $u:=(u(i,j))_{i,j\in \Zd}$ with each entry not less than one, i.e.,
\begin{equation}
u(i,j)\ge 1 \quad {\rm for \ all} \ i,j\in \Zd.\end{equation}
 For $1\le p\le \infty$ and a weight matrix $u:=(u(i,j))_{i,j\in \Zd}$, define
\begin{equation}\label{beurling.def}
{\mathcal B}_{p,u}(\Zd, \Zd):=\Big\{ A:=(a(i,j))_{i,j\in \Zd}\ \Big|\  \|A\|_{{\mathcal B}_{p,u}(\Zd, \Zd)}<\infty\Big\},
\end{equation}
where
\begin{equation}\label{beurlingnorm.def}
\|A\|_{{\mathcal B}_{p,u}(\Zd, \Zd)}:=\Big\|\Big(\sup_{|i-j|_\infty\ge |k|_\infty} |a(i,j)| u(i,j)\Big)_{k\in \Zd}\Big\|_{p}.
\end{equation}

If
 $p=1$ and  $u\equiv 1$ (i.e., all entries of the weight matrix $u$ are equal to $1$),
  then
  \begin{equation}{\mathcal B}_{p,u}(\Zd, \Zd)={\mathcal B}(\Zd, \Zd)\quad {\rm  and}\quad \|\cdot\|_{{\mathcal B}_{p,u}(\Zd, \Zd)}=\|\cdot\|_{{\mathcal B}(\Zd, \Zd)}.\end{equation}
In this paper,
we  use
  ${\mathcal B}_{p,u}, {\mathcal B}, \|\cdot\|_{{\mathcal B}_{p,u}}, \|\cdot\|_{{\mathcal B}}$
instead of  ${\mathcal B}_{p,u}(\Zd, \Zd), {\mathcal B}(\Zd, \Zd)$,  $ \|\cdot\|_{{\mathcal B}_{p,u}(\Zd, \Zd)}, \|\cdot\|_{{\mathcal B}(\Zd, \Zd)}
$ for brevity.

\begin{rem} {\rm Let $1\le p\le \infty$ and $u$ be a weighted matrix.  Define
the  {\em Gr\"ochenig-Schur class}  ${\mathcal S}_{p,u}(\Zd, \Zd)$ of infinite matrices
  by
\begin{equation}\label{schur.def}
{\mathcal S}_{p,u}(\Zd, \Zd):=\Big\{ A:=(a(i,j))_{i,j\in \Zd}\ \Big|\  \|A\|_{{\mathcal S}_{p,u}(\Zd, \Zd)}<\infty\Big\},
\end{equation}
where
\begin{equation}\label{schurnorm.def}
\|A\|_{{\mathcal S}_{p,u}(\Zd, \Zd)}:=\max\Big(
\sup_{i\in \Zd} \big\|\big(a(i,j) u(i,j)\big)_{j\in \Zd}\big\|_{p},
\sup_{j\in \Zd} \big\|\big(a(i,j) u(i,j)\big)_{i\in \Zd}\big\|_{p}\Big)
\end{equation}
\cite{gltams06, schur11, suncasp05, suntams07}.
For $p=1$, the  class ${\mathcal S}_{1,u}(\Zd, \Zd)$ were introduced by Schur \cite{schur11}  for  weight matrices $u:=(w(i)/w(j))_{i,j\in \Zd}$
generated by  positive sequences $w:=(w(i))_{i\in \Zd}$,
and
 by Gr\"ochenig and Leinert \cite{gltams06} for weight matrices
 $u:=(v(i-j))_{i,j\in \Zd}$ associated with  positive  functions $v$ on $\Rd$.
 For $1\le p\le \infty$, the  class ${\mathcal S}_{p,u}(\Zd, \Zd)$ was introduced by Sun
 for polynomial weights $u:=((1+|i-j|_\infty)^\alpha)_{i,j\in \Zd}$ with $\alpha> d(1-1/p)$ in \cite{suncasp05} and for any weighted matrix $u$ in
  \cite{suntams07}.
 From the above definition of the Gr\"ochenig-Schur class ${\mathcal S}_{p,u}(\Zd, \Zd)$, the following inclusion follows:
 \begin{equation} \label{bsinclusion.eq}
 {\mathcal B}_{p,u}(\Zd, \Zd)\subset {\mathcal S}_{p,u}(\Zd, \Zd)
 \end{equation}
 for any $1\le p\le \infty$ and  for any weight matrix $u$. }\end{rem}

 \begin{rem} {\rm  Let $1\le p\le \infty$ and $u$ be a weighted matrix.
Define
  the {\em
 Gohberg-Baskakov-Sj\"ostrand class} ${\mathcal C}_{p,u}(\Zd, \Zd)$
of infinite matrices by
\begin{equation}\label{gbs.def}
{\mathcal C}_{p,u}(\Zd, \Zd):=\Big\{ A:=(a(i,j))_{i,j\in \Zd}\ \Big|\  \|A\|_{{\mathcal C}_{p,u}(\Zd, \Zd)}<\infty\Big\},
\end{equation}
where
\begin{equation}\label{bksnorm.def}
\|A\|_{{\mathcal C}_{p,u}(\Zd, \Zd)}:=\Big\|\Big(\sup_{i-j=k} \big(|a(i,j)| u(i,j)\big)\Big)_{k\in \Zd}\Big\|_{p}
\end{equation}
\cite{baskakov90, gkwieot89, gltams06, graif08, kurbatov90, sjostrand94, suntams07}.
For $p=1$ and the trivial weight matrix $u_0$ (i.e., $u_0(i,j)=1$ for all $i,j\in \Zd$), the class ${\mathcal C}_{1,u_0}(\Zd, \Zd)$
was introduced by Gohberg, Kaashoek, and Woerdeman  \cite{gkwieot89} as a
generalization of  the class of Toeplitz matrices associated with summable sequences. It was reintroduced by Sj\"ostrand \cite{sjostrand94} in considering  algebra of pseudo-differential operators.
For $p=1$ and  nontrivial weight matrices $u:=(v(i-j))_{i,j\in \Zd}$ associated with  positive functions $v$ on $\Rd$, the  class ${\mathcal C}_{1,u}(\Zd, \Zd)$ was introduced and studied  by
Baskakov \cite{baskakov90} and Kurbatov  \cite{kurbatov90}  independently, see also \cite{gltams06}.
The above definition of the  Gohberg-Baskakov-Sj\"ostrand class ${\mathcal C}_{p,u}(\Zd, \Zd)$ is given by Sun \cite{suntams07} for any $1\le p\le \infty$ and any  weight matrix $u$.
From the  definition of the  Gohberg-Baskakov-Sj\"ostrand class ${\mathcal C}_{p,u}(\Zd, \Zd)$,
 we have  the following inclusion:
\begin{equation}\label{bcinclusion.eq}
{\mathcal B}_{p,u}(\Zd, \Zd)\subset {\mathcal C}_{p,u}(\Zd, \Zd)\end{equation}
for any $1\le p\le \infty$ and  for any weight matrix $u$.
}\end{rem}

\begin{rem}{\rm
The  inclusions \eqref{bsinclusion.eq} and \eqref{bcinclusion.eq}
become  equalities for $p=\infty$, i.e.,
\begin{equation}
{\mathcal B}_{\infty,u}(\Zd, \Zd)= {\mathcal C}_{\infty,u}(\Zd, \Zd)= {\mathcal S}_{\infty,u}(\Zd, \Zd)=:{\mathcal J}_{u}(\Zd, \Zd)\end{equation}
The class ${\mathcal J}_{u}(\Zd, \Zd)$  of infinite matrices
 is usually known as the {\em Jaffard class}, \cite{baskakov90,  dahlkejat09, grochenigklotz10,  gltams06, jaffard90, suncasp05, suntams07}. The Jaffard class ${\mathcal J}_u(\Zd, \Zd)$ with polynomial weight $u:=((1+|i-j|)^\alpha)_{i,j\in \Zd}$
 was introduced by Jaffard  \cite{jaffard90} in considering wavelets on an open domain. The Jaffard class ${\mathcal J}_u(\Zd, \Zd)$ with  weight matrices  $u:=(v(i-j))_{i,j\in \Zd}$ associated with  positive  functions $v$ on $\Rd$  was introduced  by Baskakov
 \cite{baskakov90} independently,  and later applied nontrivially in the study of
 localization of   frames \cite{gltams06}, adaptive computation \cite{dahlkejat09}, and nonuniform sampling \cite{sunsiam06}.
 }\end{rem}

\bigskip

For the class ${\mathcal B}_{p,u}$ of infinite matrices, we have the following
 proposition.

\begin{prop} \label{banachalgebra.prop}
Let $\alpha\in \CC$, $1\le p\le \infty$, $u$ be a weight matrix,
 and let
$A:=((a(i,j))_{i,j\in \Zd}$ and $B:=(b(i,j))_{i,j\in \Zd}$ belong to ${\mathcal B}_{p,u}$.
Then
\begin{itemize}
\item[{(i)}] 
    $\|A+B\|_{{\mathcal B}_{p,u}}\le \|A\|_{{\mathcal B}_{p,u}}+\|B\|_{{\mathcal B}_{p,u}}$.

\item[{(ii)}] $\|\alpha A\|_{{\mathcal B}_{p,u}}= |\alpha| \|A\|_{{\mathcal B}_{p,u}}$.

    \item[{(iii)}] $\| A^*\|_{{\mathcal B}_{p,u}}=  \|A\|_{{\mathcal B}_{p,u}}$
   where $A^*:=\big(\overline{a(j,i)}\big)_{i,j\in \Zd}$
   is the conjugate transpose of the matrix $A$.

    \item[{(iv)}] $\| A\|_{{\mathcal B}_{p,u}}\le \| B\|_{{\mathcal B}_{p,u}}$
    if  $|A|\le |B|$, i.e., $|a(i,j)|\le |b(i,j)|$ for all $ i,j\in \Zd$.

\end{itemize}

\end{prop}

All conclusions in the above  proposition follow directly from \eqref{beurling.def} and \eqref{beurlingnorm.def}.
From the conclusions (i) and (ii) of  the above proposition,  we see that
$\|\cdot\|_{{\mathcal B}_{p,u}}$ is a norm on the class ${\mathcal B}_{p,u}$ of infinite matrices. The
properties in the conclusion (iv) is usually  known as the {\em solidness} of the matrix norm $\|\cdot\|_{{\mathcal B}_{p,u}}$.

\section{Algebraic properties}
\label{algebra.section}

In  this section, we   establish
 some algebraic properties for the
class ${\mathcal B}_{p,u}$ of infinite matrices  and  give a proof of Theorem \ref{beurlingalgebra.tm}.

\bigskip

Let us first recall  the concept of
 a  $p$-submultiplicative weight matrix $u$
\cite{gltams06, suntams07, sunacha08}.
For $1\le p\le \infty$, a weight matrix  $u:=(u(i,j))_{i,j\in \Zd}$  is said to {\em $p$-submultiplicative} if
there exists  another  weight matrix $v:=(v(i,j))_{i,j\in \Zd}$  such that
\begin{equation}\label{matrix.condition0}
v(i,j)\ge 1 \quad  {\rm  for\  all}  \  i,j\in \Zd,
\end{equation}
\begin{equation}
\label{matrix.condition1}
u(i,j)\le u(i, k) v(k,j)+ v(i,k) u(k,j)\quad  {\rm  for\  all}  \ i,j, k\in \Zd,
\end{equation}
and
\begin{equation}\label{matrix.condition2}
C_{p}(v,u):=\Big\| \Big(\sup_{|i-j|_\infty\ge |k|_\infty} \big(v(i,j) (u(i,j))^{-1}\big)\Big)\Big\|_{p/(p-1)}<\infty.
\end{equation}
For  $p=1$,  we simply say that
a weight matrix is {\em submultiplicative} instead of $1$-submultiplicative.
We call the weight matrix $v$ satisfying \eqref{matrix.condition0},
\eqref{matrix.condition1} and \eqref{matrix.condition2}
 a {\em companion weight matrix} of the $p$-submultiplicative weight matrix $u$. Denote by $C(u)$ the set of all  companion weights of a $p$-submultiplicative weight matrix $u$,
and define the {\em $p$-submultiplicative bound} $M_p(u)$ by
\begin{equation}\label{matrix.condition3b}
M_p(u):=\inf_{v\in C(u)} C_p(v,u).
\end{equation}
One may verify that $C(u)$ is a convex set and the infinimum of $C_p(v,u)$ in the set $C(u)$ can be attained for
some companion weight matrix $v$. So from now on, except stated explicitly,
we {\bf always} assume that the companion weight $v$ of a
$p$-submultiplicative weight matrix $u$ is the one satisfying
\begin{equation}\label{matrix.condition3}
M_p(u)= C_p(v,u).
\end{equation}

\begin{rem}{\rm
From the   definitions of $p$-submultiplicative weight matrices on $\Zd\times \Zd$,
we  have the following:
\begin{itemize}
\item[{(i)}] A $p$-submultiplicative weight matrix is $q$-submultiplicative for all $1\le q\le p$.

\item[{(ii)}]
A necessary condition for a weight matrix
$u:=(u(i,j))_{i,j\in \Zd}$ to be $p$-submultiplicative is
$u(i,j)\le C  u(i,k) u(k,j)  $
for all $i,j,k\in \Zd$ and for some positive constant $C$. When $p=1$, the above necessary condition is also a sufficient condition \cite{gltams06}.

\item[{(iii)}] Let $1\le p\le \infty, \delta\in (0,1)$, and let $\alpha$ be a  number with the property that
$\alpha> d-d/p$ if $1< p\le \infty$,
  and $\alpha\ge 0$ if $p=1$. Then the  Toeplitz matrices
$ p_\alpha:=\big((1+|i-j|_\infty)^\alpha\big)_{i,j\in \Zd}$ 
 generated by  the polynomial weight $(1+|x|_\infty)^\alpha$,
and
$e_\delta:=\big(\exp(|i-j|_\infty^\delta)\big)_{i,j\in \Zd}$ 
  generated by the sub-exponential
  weight $\exp(|x|_\infty^\delta)$, are $p$-submultiplicative \cite{suntams07}.

\end{itemize}
}\end{rem}

Now we state the main result of this section, an extension of Theorem \ref{beurlingalgebra.tm}.

\begin{thm}\label{banachalgebra.tm}
Let $1\le p, q\le \infty$, $u$ be a $p$-submultiplicative weight matrix  with  the $p$-submultiplicative bound $M_p(u)$, and let  $w$ be a discrete  $A_q$-weight with the $A_q$-bound $A_q(w)$. Then the following statements hold.
\begin{itemize}

\item[{(i)}]  If  $v$ is a companion weight matrix of the $p$-submultiplicative weight matrix $u$, then
\begin{equation}\label{banachalgebra.tm.eq0}
\|AB\|_{{\mathcal B}_{p,u}}\le
 2^{2/p}  5^{(d-1)/p}
 \big( \|A\|_{{\mathcal B}_{p,u}}\|B\|_{{\mathcal B}_{1,v}}
 +\|A\|_{{\mathcal B}_{1,v}}\|B\|_{{\mathcal B}_{p,u}}\big)
\end{equation}
for all $A, B\in {\mathcal B}_{p,u}$.

\item[{(ii)}]
    ${\mathcal B}_{p,u}$  is (and hence ${\mathcal B}$ is also) an algebra. Moreover
\begin{equation}\label{banachalgebra.tm.eq1}
\|AB\|_{{\mathcal B}_{p,u}}\le  2^{1+2/p}  5^{(d-1)/p}  M_p(u) \|A\|_{{\mathcal B}_{p,u}}\|B\|_{{\mathcal B}_{p,u}}\quad {\rm for \ all} \  A, B\in {\mathcal B}_{p,u}.
\end{equation}

\item[{(iii)}] ${\mathcal B}_{p,u}$ is a subalgebra of ${\mathcal B}$. Moreover
    \begin{equation}\label{banachalgebra.tm.eq2}
    \|A\|_{{\mathcal B}}\le M_p(u)  \|A\|_{{\mathcal B}_{p,u}}\quad
    {\rm for \ all}
\ A\in {\mathcal B}_{p,u}.
    \end{equation}

\item[{(iv)}]
    ${\mathcal B}_{p,u}$ is  (and hence ${\mathcal B}$ is also)  a subalgebra of ${\mathcal B}(\ell^q_w)$. Moreover
    \begin{equation}\label{banachalgebra.tm.eq3}
    \|Ac\|_{q,w} \le 2^{2d}3^{d/q} (A_q(w))^{1/q} M_p(u)  \|A\|_{{\mathcal B}_{p,u}} 
    \|c\|_{q,w}
       \end{equation}
  for  all $A\in {\mathcal B}_{p,u}$ and $\ c\in \ell^q_w$.
\end{itemize}
\end{thm}

Before we  give the proof of the above theorem, let us next have some remarks on the
unital Banach algebra property of the  algebra ${\mathcal B}_{p,u}$,
on the equality of  spectral radii in the algebras ${\mathcal B}_{p,u}$  and ${\mathcal B}_{1, v}$, and on the inclusion ${\mathcal B}_{p,u}\subset {\mathcal B}(\ell^q_w)$.

\begin{rem}\label{banachalgebra.rem1}
 {\rm For $1\le p\le \infty$ and   a $p$-submultiplicative weight matrix $u:=(u(i,j))_{i,j\in \Zd}$, following the standard procedure \cite{gelfandbook, Naimarkbook} we define
$
\|A\|_{{\mathcal B}_{p,u}}^\prime:=\sup_{\|B\|_{{\mathcal B}_{p,u}}=1}
\|AB\|_{{\mathcal B}_{p,u}}$ for $A\in {\mathcal B}_{p,u}$.
  Then
 \begin{equation}\label{banachalgebra.rem1.eq2}
\|AB\|_{{\mathcal B}_{p,u}}^\prime\le \|A\|_{{\mathcal B}_{p,u}}^\prime\|B\|_{{\mathcal B}_{p,u}}^\prime\quad{\rm for \ all} \ A,B\in {\mathcal B}_{p,u}.
\end{equation}
 If
the weight matrix $u$  further satisfies
\begin{equation}\label{banachalgebra.rem1.eq3}
M:=\sup_{i\in \Zd} u(i,i)<\infty,
\end{equation}
then the identity matrix $I$ belongs to ${\mathcal B}_{p,u}$,
 and the norms $\|\cdot\|_{{\mathcal B}_{p,u}}$ and
 $\|\cdot\|_{{\mathcal B}_{p,u}}^\prime$ on ${\mathcal B}_{p,u}$ are equivalent to each other, because
\begin{equation*}\label{banachalgebra.rem1.eq4}
M^{-1} \|A\|_{{\mathcal B}_{p,u}}\le
\|A\|_{{\mathcal B}_{p,u}}^\prime\le 2^{1+2/p}  5^{(d-1)/p}  M_p(u)\|A\|_{{\mathcal B}_{p,u}}
\quad {\rm for \ all} \ A\in {\mathcal B}_{p,u}\end{equation*}
by the conclusion (ii) of Theorem \ref{banachalgebra.tm}
and the fact that $\|I\|_{{\mathcal B}_{p,u}}=M$.
Therefore if $1\le p\le \infty$ and $u$ is a $p$-submultiplicative weight matrix
satisfying \eqref{banachalgebra.rem1.eq3}, then
the class ${\mathcal B}_{p,u}$ of infinite matrices endowed with the norm
$\|\cdot\|_{{\mathcal B}_{p,u}}^\prime$ becomes a  unital Banach algebra.
}\end{rem}

\begin{rem}\label{banachalgebra.rem2}{\rm Let $1\le p\le \infty$, $u$ be
a $p$-submultiplicative weight matrix, and $v$ be its companion weight matrix.
If the companion weight matrix $v$ is submultiplicative, then both
${\mathcal B}_{p,u}$ and ${\mathcal B}_{1,v}$ are algebras by the conclusion (ii)
of Theorem \ref{banachalgebra.tm},
and ${\mathcal B}_{p,u}$  is a subalgebra of ${\mathcal B}_{1,v}$
since
\begin{equation}\label{banachalgebra.rem2.eq1}
\|A\|_{{\mathcal B}_{1,v}}\le C_p(v,u) \|A\|_{{\mathcal B}_{p,u}}\quad {\rm for \ all} \ A\in {\mathcal B}_{p,u}.
\end{equation}
Applying \eqref{banachalgebra.rem2.eq1} with $A$ replaced by $A^n$ and then
taking $n$-th roots and the limit as $n\to\infty$ yields
\begin{equation*}\label{banachalgebra.rem2.eq2}
\rho_{{\mathcal B}_{1,v}}(A):=\limsup_{n\to \infty}
(\|A^n\|_{{\mathcal B}_{1,v}})^{1/n} \le \limsup_{n\to \infty}
(\|A^n\|_{{\mathcal B}_{p,u}})^{1/n}=:\rho_{{\mathcal B}_{p,u}}(A).
\end{equation*}
Applying
the  conclusion (i)  of Theorem \ref{banachalgebra.tm} gives
\begin{equation} \label{banachalgebra.rem2.eq3}
\|A^{2n}\|_{{\mathcal B}_{p,u}}
\le  2^{1+2/p} 5^{(d-1)/p}\|A^n\|_{{\mathcal B}_{p,u}}\|A^n\|_{{\mathcal B}_{1,v}},
\end{equation}
and then taking $n$-th roots  and   letting $n\to\infty$ lead to
the inequality
$\rho_{{\mathcal B}_{p,u}}(A)
\le \rho_{{\mathcal B}_{1,v}}(A)$.
This concludes that if  $u$ is
a $p$-submultiplicative weight matrix and  its companion weight matrix
$v$ is submultiplicative, then
 the spectral radii $\rho_{{\mathcal B}_{p,u}}(A)$
 and
 $\rho_{{\mathcal B}_{1,v}}(A)$
 are the same for any $A\in {\mathcal B}_{p,u}$, i.e.,
$\rho_{{\mathcal B}_{1,v}}(A)= \rho_{{\mathcal B}_{p,u}}(A)$  for  all $A\in {\mathcal B}_{p,u}$.
The above procedure to
establish the equality of spectral radii in the algebras ${\mathcal B}_{p,u}$ and ${\mathcal B}_{1,v}$ from  the inequality in  the conclusion (i) of Theorem \ref{banachalgebra.tm}
 is  known as {\em  Brandenburg's trick} \cite{brandenburg75, grochenigklotz10}.
 Another technique to
prove  the equality of spectral radii in two algebras ${\mathcal A}_1$ and
 ${\mathcal A}_2$ with the same identity is by showing that
  \begin{equation}\label{banachalgebra.rem2.eq6}
  \|A\|_{{\mathcal A}_2}\le C \|A\|_{{\mathcal A}_1}
  \end{equation}
 and
 \begin{equation}\label{banachalgebra.rem2.eq7}
 \|A^2\|_{{\mathcal A}_1}\le C \|A\|_{{\mathcal A}_1}^{1+\theta} \|A\|_{{\mathcal A}_2}^{1-\theta}\quad {\rm for \ all} \ A\in {\mathcal A}_1,
 \end{equation}
 where  $\|\cdot\|_{{\mathcal A}_1}$ and
 $\|\cdot\|_{{\mathcal A}_2}$ are norms in the algebra ${\mathcal A}_1$ and ${\mathcal A}_2$ respectively, and where $C\in (0, \infty)$ and $\theta\in [0,1)$ are  constants
 independent of $A\in {\mathcal A}$. The estimates in \eqref{banachalgebra.rem2.eq6} and \eqref{banachalgebra.rem2.eq7}  for ${\mathcal A}_2={\mathcal B}(\ell^2)$ and
 ${\mathcal A}_1={\mathcal S}_{p,u}(\Zd, \Zd)$ or ${\mathcal C}_{p,u}(\Zd, \Zd)$ are established in
 \cite{suncasp05, suntams07}, while the ones for ${\mathcal A}_2={\mathcal B}(\ell^2)$ and
 ${\mathcal A}_1={\mathcal B}_{p,u}, 1\le p\le \infty$, are given in Lemma \ref{wienerlemmaell2.lem2}.
}\end{rem}

\begin{rem} {\rm
 The conclusion (iv) of Theorem \ref{banachalgebra.tm}  about the boundedness of  an infinite matrix in ${\mathcal B}$  on the weight sequence space $\ell^q_w$ is a simplified discrete version of  the second conclusion  in \cite[Proposition 2 of Chapter 10]{steinbook93}.
The reader may refer to \cite[Lemma 3.1]{graif08} for a general result on the  boundedness of an infinite matrix on  sequence spaces.
}\end{rem}

We conclude this section  by  giving the proof of  Theorem \ref{banachalgebra.tm}.

\begin{proof}[Proof of Theorem \ref{banachalgebra.tm}] 
 {\bf (i)}:\quad
Let  $1\le p\le \infty$, $u$ be a $p$-submultiplicative weight matrix, and let
$v$ be a companion weight matrix of the weight matrix $u$. Take
$A:=(a(i,j))_{i,j\in \Zd}$ and $B:=(b(i,j))_{i,j\in \Zd}$  in  ${\mathcal B}_{p,u}$, and write $AB:=(c(i,j))_{i,j\in \Zd}$.
Then  it follows from \eqref{matrix.condition1} that
\begin{eqnarray}\label{banachalgebra.pf.eq1}  |c(i,j)| u(i,j)   & = &
\Big| \sum_{k\in \Zd} a(i,k) b(k,j)\Big| u(i,j)\nonumber\\
& \le  &
 \sum_{k\in \Zd} |a(i,k)| u(i,k)  |b(k,j)| v(k,j)\nonumber\\
& & +
\sum_{k\in \Zd} |a(i,k)| v(i,k) |b(k,j)| u(k,j)
\quad {\rm for\ all} \  i,j\in \Zd.
\end{eqnarray}
For $1\le p<\infty$, we obtain from \eqref{banachalgebra.pf.eq1} that
\begin{eqnarray*} 
 & & \sum_{k\in \Zd} |a(i,k)| u(i,k)  |b(k,j)| v(k,j)\nonumber\\
& \le & \Big(\sum_{k'\in \Zd} (|a(i,k')| u(i,k'))^p  |b(k',j)| v(k',j)\Big)^{1/p}\nonumber\\
& & \quad \times \Big(\sum_{k^{\prime\prime}\in \Zd} |b(k^{\prime\prime},j)|v(k^{\prime\prime},j)\Big)^{(p-1)/p}\nonumber\\
& \le &
\big( \|B\|_{{\mathcal B}_{1,v}}\big)^{(p-1)/p} \Big\{
\Big(\sum_{k'\in \Zd}  |b(k',j)| v(k',j)\Big)
\nonumber\\
& & \qquad\quad\times
\Big(\sup_{|i'-j'|_\infty\ge |i-j|_\infty/2}
\big(|a(i',j')| u(i',j')\big)^p \Big)
\nonumber\\
& & \qquad  +
 \Big(\sup_{|i'-j'|_\infty\ge |i-j|_\infty/2}
\big(|b(i',j')| v(i',j')\big) \Big)\nonumber\\
& & \qquad \quad \times
\Big(\sum_{k'\in \Zd}  \big(|a(i, k')| u(i, k')\big)^p\Big)
\Big\}^{1/p}\nonumber\\
& \le & \big(\|B\|_{{\mathcal B}_{1,v}}\big)^{(p-1)/p}\Big\{ \|B\|_{{\mathcal B}_{1,v}}\Big(\sup_{|i'-j'|_\infty\ge |i-j|_\infty/2}
\big(|a(i',j')| u(i',j')\big)^p \Big) \nonumber\\
& & \quad  +\big( \|A\|_{{\mathcal B}_{p,u}}\big)^p
 \Big(\sup_{|i'-j'|_\infty\ge |i-j|_\infty/2}
\big(|b(i',j')| v(i',j')\big) \Big)
\Big\}^{1/p},\end{eqnarray*}
and
\begin{eqnarray*} 
& & \sum_{k\in \Zd} |a(i,k)| v(i,k) |b(k,j)| u(k,j)\\
& \le &
\big(\|A\|_{{\mathcal B}_{1,v}}\big)^{(p-1)/p}\Big\{ \|A\|_{{\mathcal B}_{1,v}}\Big(\sup_{|i'-j'|_\infty\ge |i-j|_\infty/2}
\big(|b(i',j')| u(i',j')\big)^p \Big) \nonumber\\
& & \quad  + \big(\|B\|_{{\mathcal B}_{p,u}}\big)^p
 \Big(\sup_{|i'-j'|_\infty\ge |i-j|_\infty/2}
\big(|a(i',j')| v(i',j')\big) \Big)
\Big\}^{1/p}.
\end{eqnarray*}
Combining the above two estimates with
\eqref{banachalgebra.pf.eq1} leads to
\begin{eqnarray*} 
 \|AB\|_{{\mathcal B}_{p,u}}
&= &
\Big\|\Big(\sup_{|i-j|_\infty\ge |k|_\infty} \big(|c(i,j)| u(i,j)\big)\Big)_{k\in \Zd}\Big\|_p\nonumber\\
& \le &  \big( \|B\|_{{\mathcal B}_{1,v}}\big)^{(p-1)/p} %
 \Big\{  \|B\|_{{\mathcal B}_{1,v}}
 \nonumber\\
 & & \qquad \times
 \Big\|\Big(\sup_{|i-j|_\infty\ge |k|_\infty/2}\big (|a(i',j')| u(i',j')\big)^p  \Big)_{k\in \Zd}
\Big\|_1\nonumber\\
 & &\quad  + \big(\|A\|_{{\mathcal B}_{p,u}}\big)^p
 \Big\|\Big(\sup_{|i-j|_\infty\ge |k|_\infty/2} \big(|b(i',j')| v(i',j')\big)  \Big)_{k\in \Zd}
\Big\|_1
\Big\}^{1/p}
 \nonumber\\
& & + \big(\|A\|_{{\mathcal B}_{1,v}}\big)^{(p-1)/p} \Big\{\|A\|_{{\mathcal B}_{1,v}}\nonumber\\
 & & \qquad \times
 \Big\|\Big(\sup_{|i-j|_\infty\ge |k|_\infty/2} \big(|b(i',j')| u(i',j')\big)^p  \Big)_{k\in \Zd}
\Big\|_1\nonumber\\
 & &\quad  + \big(\|B\|_{{\mathcal B}_{p,u}}\big)^p
 \Big\|\Big(\sup_{|i-j|_\infty\ge |k|_\infty/2} \big(|a(i',j')| v(i',j')\big)  \Big)_{k\in \Zd}
\Big\|_1
\Big\}^{1/p}
 \nonumber\\
 & \le &
  2^{2/p}  5^{(d-1)/p}
  \Big( \|A\|_{{\mathcal B}_{p,u}}\|B\|_{{\mathcal B}_{1,v}}+
  \|A\|_{{\mathcal B}_{1,v}}\|B\|_{{\mathcal B}_{p,u}}\Big),
  \end{eqnarray*}
  where we have used
  the fact that
  \begin{equation} \label{banachalgebra.pf.eq6}
 \Big\|\Big(\sup_{|i-j|_\infty\ge |k|_\infty/N} |a(i,j)|\Big)\Big\|_1\le
 N  (2N+1)^{d-1}  \Big\|\Big(\sup_{|i-j|_\infty\ge |k|_\infty} |a(i,j)|\Big)\Big\|_1
 \end{equation}
 for any integer $N\ge 1$ and $A:=(a(i,j))\in {\mathcal B}$.
This   proves \eqref{banachalgebra.tm.eq0} for $1\le p<\infty$.

For $p=\infty$, it follows from \eqref{banachalgebra.pf.eq1}  that
\begin{equation}\label{banachalgebra.pf.eq2}
\|AB\|_{{\mathcal B}_{\infty,u}}  \le
\|A\|_{{\mathcal B}_{\infty,u}}\|B\|_{{\mathcal B}_{1,v}}+
\|A\|_{{\mathcal B}_{1,v}}\|B\|_{{\mathcal B}_{\infty,u}}.
\end{equation}
Hence \eqref{banachalgebra.tm.eq0} for $p=\infty$ is proved.

\bigskip
   {\bf (ii)}\quad 
Let  $v$ be the companion weight matrix of the $p$-submultiplicative weight $u$ that
satisfies \eqref{matrix.condition0}--\eqref{matrix.condition2} and \eqref{matrix.condition3}. Then
\begin{equation}\label{banachalgebra.pf.eq7-a}
\|A\|_{{\mathcal B}_{1, v}}\le M_p(u)\|A\|_{{\mathcal B}_{p, u}}\quad {\rm for \ all} \ A\in {\mathcal B}_{p,u},
\end{equation}
because
\begin{eqnarray}\label{banachalgebra.pf.eq7}
  \sup_{|i-j|_\infty\ge |k|_\infty} |a(i,j)| v(i,j) 
& \le &
\Big(\sup_{|i'-j'|_\infty\ge |k|_\infty} |a(i',j')| u(i',j')\Big)\nonumber\\
& &  \times  \Big(\sup_{|i'-j'|_\infty\ge |k|_\infty} v(i',j') (u(i',j'))^{-1}\Big)
\end{eqnarray}
hold for all $k\in \Zd$.
Combining  \eqref{banachalgebra.tm.eq0} and \eqref{banachalgebra.pf.eq7-a}
proves \eqref{banachalgebra.tm.eq1}.

\bigskip
{\bf (iii)}\quad
Let  $v$ be the companion weight matrix of the $p$-submultiplicative weight $u$ that
satisfies \eqref{matrix.condition0}--\eqref{matrix.condition2} and \eqref{matrix.condition3}.
Then
\begin{equation}
\|A\|_{{\mathcal B}}\le \|A\|_{{\mathcal B}_{1, v}}\quad {\rm for \ all} \ A\in
{\mathcal B}_{1, v}
\end{equation}
by \eqref{matrix.condition0} for the weight matrix $v$.
This together with \eqref{banachalgebra.pf.eq7-a}
gives \eqref{banachalgebra.tm.eq2} and hence proves
the conclusion (iii).

\bigskip

{\bf (iv)}\quad By  (iii), it suffices to prove
\begin{equation}\label{banachalgebra.pf.eq8a}
\|Ac\|_{q,w}\le 2^{2d} 3^{d/q} (A_q(w))^{1/q} \|A\|_{\mathcal B} \|c\|_{q,w}
\end{equation}
for all $A:=(a(i,j))_{i,j\in \Zd}\in {\mathcal B}$ and $c\in \ell^q_w$.
Set $h(n):=\sup_{|i-j|_\infty\ge n} |a(i,j)|$.
Then $\{h(n)\}_{n=0}^\infty$ is a decreasing sequence, i.e., $h(n+1)\le h(n)$ for all $n\ge 0$, and
\begin{eqnarray} \label{banachalgebra.pf.eq9b}
  \sum_{l=1}^\infty h(2^{l-1}) 2^{(l+1)d}
& \le  & 2^{2d} h(1)+ 2^{d+2}
\sum_{l=2}^\infty\Big(\sum_{2^{l-2}<s\le 2^{l-1}} h(s)\Big) 2^{l(d-1)}\nonumber\\
& \le & 2^{2d} h(1)+ 2^{3d}\sum_{s=2}^\infty h(s)s^{d-1} 
\nonumber\\
& \le & 2^{2d} h(1)+  2^{2d} d^{-1} \sum_{s=2}^\infty
\sum_{ k\in \Zd \ {\rm with} \ |k|_\infty=s}  h(|k|_\infty)\nonumber\\
&\le &   2^{2d} (\|A\|_{{\mathcal B}}-h(0)).
\end{eqnarray}

For $1<q<\infty$ and a discrete $A_q$-weight $w$,
\begin{eqnarray*}
\|Ac\|_{q,w} & \le &
\Big\{\sum_{i \in \Zd} \Big(\sum_{j\in \Zd}  h(|i-j|_\infty) |c(j)|\Big)^q w(i)\Big\}^{1/q}\nonumber\\
& \le &   h(0)
 \Big\{\sum_{i\in \Zd} |c(i)|^q w(i)\Big\}^{1/q}
 + \Big\{
\sum_{i\in \Zd} w(i)
\Big(\sum_{l=1}^\infty h(2^{l-1}) 2^{(l+1)d} \Big)^{q-1}
\nonumber\\
& &\ \times
 \Big(\sum_{l=1}^\infty h(2^{l-1}) 2^{-(l+1)d(q-1)}
\Big( \sum_{2^{l-1}\le |i-j|_\infty< 2^l} |c(j)|\Big)^q\Big) \Big\}^{1/q}.
\end{eqnarray*}
Thus
\begin{eqnarray*}\\
\|Ac\|_{q,w} & \le &
h(0) \|c\|_{q,w}
  + 2^{2d(q-1)/q} (\|A\|_{\mathcal B}-h(0))^{(q-1)/q}\nonumber\\
& & \quad \times
\Big\{\sum_{i\in \Zd}  \sum_{l=1}^\infty  w(i) h(2^{l-1}) 2^{-(l+1)d(q-1)}
\nonumber\\
 & &\qquad \times
\Big(\sum_{2^{l-1}\le |i-j|_\infty< 2^l} |c(j)|^q w(j)\Big)\nonumber\\
& & \qquad
\times
\Big(\sum_{2^{l-1}\le |i-j'|_\infty< 2^l} (w(j'))^{-1/(q-1)}\Big)^{q-1}\Big\}^{1/q}.
\end{eqnarray*}
This together with the discrete $A_q$-weight assumption leads to
\begin{eqnarray*}
\|Ac\|_{q,w} & \le &
 h(0) \|c\|_{q,w} +2^{2d(q-1)/q} (\|A\|_{\mathcal B}-h(0))^{(q-1)/q} (A_q(w))^{1/q} 
 \nonumber\\
 & & \times \Big\{
 \sum_{l=1}^\infty   h(2^{l-1}) 2^{(l+1)d} \sum_{i\in \Zd} \frac{w(i)}{\sum_{|i-j'|_\infty<2^l} w(j')}\nonumber\\
 & & \qquad \times
 \Big(\sum_{2^{l-1}\le |i-j|_\infty< 2^l} |c(j)|^q w(j)\Big) \Big\}^{1/q}\nonumber\\
& \le & h(0) \|c\|_{q,w}
 +2^{2d(q-1)/q} (\|A\|_{\mathcal B}-h(0))^{(q-1)/q} (A_q(w))^{1/q}
 \nonumber\\
 & & \times    \Big\{
 \sum_{l=1}^\infty   h(2^{l-1}) 2^{(l+1)d}
 \Big( \sum_{j\in \Zd} |c(j)|^q w(j)\nonumber\\
 & & \quad
\Big(\sum_{\epsilon\in \{-1, 0, 1\}^d} \sum_{|i-j-\epsilon 2^{l-1}|_\infty<2^{l-1}}\frac{w(i)}
{\sum_{|i-j'|_\infty<2^l} w(j')}
\Big)\Big)\Big\}^{1/q}\nonumber\\
& \le & 2^{2d} 3^{d/p} (A_q(w))^{1/q} \|A\|_{\mathcal B}\|c\|_{q,w},
\end{eqnarray*}
and hence \eqref{banachalgebra.tm.eq3} for $1<q<\infty$ is established.

The conclusion \eqref{banachalgebra.tm.eq3} for $q=1$ can be proved by similar arguments. We omit the details here.
 \end{proof}

\section{$\ell^q_w$-stability}
\label{stability.section}

In this section, we  prove   the following  theorem, a slight generalization of Theorem \ref{beurlingstability.tm},
and  Corollary \ref{leftinverse.cor1}. We also provide a characterization to
the $\ell^q_w$-stability of a Toeplitz matrix in ${\mathcal B}$.

\begin{thm}\label{sequencespacesequivalence.thm} Let $1\le p\le \infty$,
$1\le q, q'<\infty$, let the weight matrix $u$ be  $p$-submultiplicative,  and $w, w'$ be  discrete $A_{q}$-weight and $A_{q'}$-weight respectively.
If $A\in {\mathcal B}_{p,u}$ has $\ell^{q}_{w}$-stability,
then $A$ has $\ell^{q'}_{w'}$-stability.
\end{thm}

As  the trivial weight $w_0$ (i.e. $w_0(i)=1$ for all $i\in \Zd$)  is a discrete $A_q$-weight for any $1\le q<\infty$, we have the following corollary
of Theorem \ref{sequencespacesequivalence.thm}. Similar result about $\ell^q$-stability for different exponents $q\in [1, \infty]$
is  established by Aldroubi, Baskakov and Krishtal \cite{akramjfa09} for infinite matrices in the  Gohberg-Baskakov-Sj\"ostrand class ${\mathcal C}_{1, p_\alpha} (\Zd, \Zd)$ with $\alpha>(d+1)^2$, by Tessera \cite{tessera09} for $\alpha>0$, and  by Shin and Sun \cite{shincjfa09} for $\alpha\ge 0$, where $p_\alpha=((1+|i-j|_\infty)^\alpha)_{i,j\in \Zd}$.

\begin{cor}\label{sequencespacesequivalence.cor}
If $A\in {\mathcal B}$ has $\ell^{q}$-stability for some $1\le q<\infty$,
then $A$ has $\ell^{q'}$-stability for all $1\le q'<\infty$.
\end{cor}

The $\ell^q_w$-stability  is one of the basic assumptions for infinite
matrices arising in many  fields of mathematics
(see \cite{akramgrochenig01,
balanchl04,  grochenigl03, grochenigstrohmer, jaffard90,   sjostrand94, sunsiam06,
suntams07}
for a sample of papers), but
little is known about
practical criteria for the $\ell^q_w$-stability of an infinite
matrix,  see \cite{sunappear} for the  diagonal-blocks-dominated criterion for the  $\ell^2$-stability of infinite matrices in the Gohberg-Baskakov-Sj\"ostrand class ${\mathcal C}(\Zd, \Zd)$.
As an application of Theorem  \ref{beurlingstability.tm}, we have the following characterization to the $\ell^q_w$-stability of a Toeplitz matrix in ${\mathcal B}$.

\begin{cor} \label{Toeplitzcriterion.cor}
Let $1\le q< \infty$, $A:=(a(i-j))_{i,j\in \Zd}$ be a Toeplitz matrix in ${\mathcal B}$, and let $w$ be a discrete $A_q$-weight. Then $A$ has
$\ell^q_w$-stability if and only if
 $\hat a(\xi):=\sum_{n\in \Zd} a(n) e^{-\sqrt{-1}\ n\xi}\ne 0$ for all $\xi\in \Rd$.
\end{cor}

To prove Theorem \ref{sequencespacesequivalence.thm}, we recall a characterization for discrete $A_q$-weights.

\begin{lem} {\rm \cite{garciabook85, steinbook93}}\  Let $1\le q<\infty$. Then  $w:=(w(i))_{i\in \Zd}$ is a discrete $A_q$-weight with the $A_q$-bound $A_q(w)$ if and only if
\begin{eqnarray}\label{discreteaqweight.eq4}
 & & \Big( N^{-d} \sum_{i\in a+[0,N-1]^d} |c(i)| \Big)^q
\Big(N^{-d} \sum_{i\in a+[0,N-1]^d} w(i) \Big)\nonumber\\
&  \le &
A_q(w)N^{-d} \sum_{i\in a+[0,N-1]^d}  |c(i)|^q w(i) \end{eqnarray}
hold for all $a\in \Zd, 1\le N\in \ZZ$ and  sequences $c:=(c(i))_{i\in \Zd}$.
\end{lem}

To prove Theorem \ref{sequencespacesequivalence.thm},
we  need a technical lemma about estimating  a bounded sequence $c$ via the sequence $Ac$, which will also be  used later in the proof of Theorem \ref{wienerlemmaforbeurlingalgebra.tm}.
Similar estimate  is given in \cite{sjostrand94} when the infinite matrix  $A$ belongs to the Gohberg-Baskakov-Sj\"ostrand class ${\mathcal C}(\Zd, \Zd)$ and has $\ell^p_w$-stability for the trivial weight $w\equiv 1$.

\begin{lem}\label{sequencespacesequivalence.lem}
Let
$1\le q<\infty$,   and $w$ be  a discrete  $A_{q}$-weight.
If $A\in {\mathcal B}$ has $\ell^{q}_{w}$-stability,
then there exists a  nonnegative sequence $\{g(i)\}_{i\in \Zd}$ on $\Zd$ such that
\begin{equation}\label{sequencespacesequivalence.lem.eq1}
\sum_{k\in \Zd} \Big(\sup_{|i|_\infty \ge |k|_\infty} g(i)\Big)<\infty
\end{equation}
and
\begin{equation}\label{sequencespacesequivalence.lem.eq2}
|c(i)|\le \sum_{j\in \Zd} g(i-j) |(Ac)(j)|, \ i\in \Zd,
\end{equation}
where $c\in \ell^\infty$.
\end{lem}

\begin{proof} 
 Without loss of generality, we assume that
\begin{equation}\label{sequencespacesequivalence.tm.pf.eq1}
\|c\|_{q,w}\le \|Ac\|_{q,w}\quad {\rm for   \ all} \ c\in \ell^q_w.
\end{equation}
Let $h(x)=\min(\max(2-|x|_\infty,0),1)$  and  $N$ be a sufficiently large  integer chosen later.
Define
 linear  operators  $\Psi_n^N, n\in N\Zd$, on $\ell^q_w$ by
 \begin{equation}
\label{multiplicationoperator.def} \Psi_n^N
c:=\Big(h\big(\frac{j-n}{N}\big)
c(j)\Big)_{j\in \Zd} \quad {\rm for}\  c:=(c(j))_{j\in \Zd}\in \ell^q_w.
\end{equation}
Then
 for $c:=(c(j))_{j\in \Zd}\in \ell^q_w$ and $|n-n'|_\infty\le 8N$,
 \begin{eqnarray} \label{sequencespacesequivalence.tm.pf.eq3}
  & & \|(\Psi_n^N A- A \Psi_n^N)\Psi_{n'}^Nc\|_{q,w}\nonumber\\
 & = &
 \Big\{\sum_{i\in \Zd} \Big| \sum_{j\in \Zd}
  \Big( h\big(\frac{i-n}{N}\big)-h\big(\frac{j-n}{N}\big)\Big)\nonumber\\
   & & \quad \times a(i,j) h\big(\frac{j-n'}{N}\big)
c(j) \Big|^q w(i)\Big\}^{1/q}\nonumber\\
& \le & N^{-1/2}  \Big\{ \sum_{i\in \Zd} \Big ( \sum_{|i-j|_\infty\le \sqrt{N}}
 |a(i,j)| |c(j)| \Big )^q w(i)\Big\}^{1/q}\nonumber\\
 & &  +
\Big\{  \sum_{i\in \Zd} \Big ( \sum_{|i-j|_\infty>\sqrt{N}}
 |a(i,j)| |c(j)| \Big )^q w(i)\Big\}^{1/q}\nonumber\\
 & \le &
 \Big\{2^{2d+2d/q} N^{-1/2}  (A_q(w))^{1/q}
  \|A\|_{\mathcal B}
 +   2^{3d+2d/q+1} (A_q(w))^{1/q}    \nonumber\\
 & &
 \quad \times \Big(\sum_{|k|_\infty\ge \sqrt{N}/2}
 \sup_{|i-j|_\infty\ge|k|_\infty} |a(i,j)|\Big)\Big\} \|c\|_{q,w},
\end{eqnarray}
where  the last inequality  follows from
 Theorem \ref{banachalgebra.tm} and the following estimate:
\begin{eqnarray*}
 & & \sum_{k\in \Zd} \sup_{|i-j|_\infty\ge \max(|k|_\infty, \sqrt{N})}
|a(i,j)|\\
 & \le &
(2\sqrt{N}+1)^d\sup_{|i-j|_\infty\ge \sqrt{N}}|a(i,j)|
 +
 \sum_{|k|_\infty> \sqrt{N}}
 \sup_{|i-j|_\infty\ge|k|_\infty} |a(i,j)|\\
  & \le & 2^{d+1}  \sum_{|k|_\infty\ge \sqrt{N}/2}
 \sup_{|i-j|_\infty\ge|k|_\infty} |a(i,j)|.
 \end{eqnarray*}
Similarly for  $c:=(c(j))_{j\in \Zd}\in \ell^q_w$ and $|n-n'|_\infty>8N$,
\begin{eqnarray} \label{sequencespacesequivalence.tm.pf.eq4}
& &  \|(\Psi_n^N A- A \Psi_n^N)\Psi_{n'}^Nc\|_{q,w}\nonumber\\
& = & \Big(\sum_{i\in \Zd} \Big| \sum_{j\in \Zd}
 h\big(\frac{i-n}{N}\big) a(i,j) h\big(\frac{j-n'}{N}\big) c(j)\Big|^q w(i) \Big)^{1/q}\nonumber\\
& \le &
 \Big(\sup_{|i'-j'|_\infty\ge |n-n'|_\infty/2} |a(i',j')| \Big)
\Big(\sum_{|i-n|_\infty<2N} \Big ( \sum_{|j-n'|_\infty<2N}
 |c(j)|\Big )^q w(i) \Big)^{1/q}
\nonumber\\
& \le &  2^{2d} N^d  (A_q(w))^{1/q}
 \Big(\sup_{|i'-j'|_\infty\ge |n-n'|_\infty/2} |a(i',j')| \Big)\nonumber\\
& & \qquad \qquad\qquad \qquad \qquad \times
\Big(\frac{ \sum_{|i'-n|_\infty<2N}  w(i')}
{ \sum_{|i'-n'|_\infty<2N}  w(i')}\Big)^{1/q}
 \|c\|_{q,w}.
\end{eqnarray}

Define
\begin{equation} \label{sequencespacesequivalence.tm.pf.eq5}
\alpha_n:= \sum_{|i'-n|_\infty<2N}  w(i'), \quad n\in N\Zd.
\end{equation}
and  the linear operator  $\Phi_N$ on $\ell^p_w$  by
 \begin{equation}
 \Phi_Nc:=\Big(\Big(\sum_{n\in N\Zd}\Big( h\big(\frac{j-n}{N}\big)\Big)^2\Big)^{-1} c(j)\Big)_{j\in \Zd} \ {\rm for} \ c:=(c(j))_{j\in \Zd}\in \ell^p_w.
 \end{equation}
 Then
for all $n'\in N\Zd$ with $|n-n'|\le 8N$,
\begin{equation}\label{sequencespacesequivalence.tm.pf.eq6}
\alpha_n\le  \sum_{|i'-n'|_\infty<10N}  w(i')
\le  6^{dq} A_q(w) \alpha_{n'}
\end{equation}
by \eqref{discreteaqweight.eq4}, and
\begin{equation}\label{sequencespacesequivalence.tm.pf.eq6+}
\|\Phi_N c\|_{q, w} \le \|c\|_{q,w} \quad {\rm for \ all}  \ c\in \ell^q_w.\end{equation}
 Note that $\Psi_n^Nc\in \ell^p_w$ for any $c\in \ell^\infty$ and $n\in N\Zd$, and
 \begin{equation}\label{sequencespacesequivalence.tm.pf.eq6++}
\|\Psi_n^Nc\|_{q,w}\le  \alpha_n^{1/q} \|c\|_\infty, \ n\in N\Zd. \end{equation}
 Then for $c\in \ell^\infty$, combining \eqref{sequencespacesequivalence.tm.pf.eq1}, \eqref{sequencespacesequivalence.tm.pf.eq3},
\eqref{sequencespacesequivalence.tm.pf.eq4}, \eqref{sequencespacesequivalence.tm.pf.eq6}, and
\eqref{sequencespacesequivalence.tm.pf.eq6+}
 leads to
\begin{eqnarray}
\label{sequencespacesequivalence.tm.pf.eq7}
 & & \quad \alpha_n^{-1/q} \|\Psi_n^N c\|_{q,w}    \le     \alpha_n^{-1/q} \| A \Psi_n^N c\|_{q,w}\\ 
 & \le & \alpha_n^{-1/q}  \|  \Psi_n^N A c\|_{q,w}+ \alpha_n^{-1/q}
 \| (\Psi_n^N A-A\Psi_n^N) c\|_{q,w}\nonumber\\
 & \le & \alpha_n^{-1/q}   \|  \Psi_n^N A c\|_{q,w}+ \alpha_n^{-1/q}
\sum_{n'\in N\Zd} \big \| (\Psi_n^N A-A\Psi_n^N) \Psi_{n'}^N \Phi_N \Psi_{n'}^N c\big\|_{q,w}\nonumber\\
 & \le &  \alpha_n^{-1/q}  \|  \Psi_n^N A c\|_{q,w}+ 2^{2d+2d/q} 6^d  (A_q(w))^{2/q}
\sum_{|n'-n|_\infty\le 8N}  \alpha_{n'}^{-1/q} \|\Psi_{n'}^Nc\|_{q,w}
 \nonumber\\  & & \quad
 \times\Big \{
\Big (  N^{-1/2} \|A\|_{{\mathcal B}}
  +  2^{d+1}
 \sum_{|k_\infty\ge  \sqrt{N}/2}
 \sup_{|i'-j'|_\infty\ge|k|_\infty} |a(i',j')| \Big) \Big\}\nonumber\\
 & &\quad
 +
  \sum_{|n'-n|_\infty> 8N}
   \alpha_{n'}^{-1/q} \| \Psi_{n'}^N c\|_{q,w}\nonumber\\
   & & \quad \times \Big\{ 2^{2d} N^d (A_q(w))^{1/q}
   \Big(\sup_{|i'-j'|_\infty\ge |n-n'|_\infty/2} |a(i',j')| \Big)\Big\}
  \nonumber\\
 & =: &   \alpha_n^{-1/q}\|  \Psi_n^N A c\|_{q,w}+
  \sum_{n'\in N\Zd} V_N(n-n') \alpha_{n'}^{-1/q}\| \Psi_{n'}^N c\|_{q,w}.
  \nonumber
\end{eqnarray}

Define  sequences  $V^l_N:=(V^l_N(n))_{n\in N\Zd}, l\ge 1$, as follows:
\begin{equation} \label{sequencespacesequivalence.tm.pf.eq9}
\left\{\begin{array}{l}
V^l_N(n):= V_N(n) \quad {\rm if} \ l=1 \ {\rm and} \  n\in N\Zd, \\
V^l_N(n):=\sum_{n'\in N\Zd} V_N(n-n') V_N^{l-1}(n') \quad {\rm if} \ l\ge 2 \ {\rm and}\  n\in N\Zd.
\end{array}\right.
\end{equation}
 Then for $c\in \ell^\infty$, applying \eqref{sequencespacesequivalence.tm.pf.eq7} repeatedly yields
 \begin{eqnarray} \label{sequencespacesequivalence.tm.pf.eq14}
 & &\alpha_n^{-1/q}\|\Psi_n^N c\|_{q,w}\nonumber\\
  &\le &   \alpha_n^{-1/q}\|\Psi_n^N Ac\|_{q,w}+\sum_{l=1}^{l_0}
\sum_{n'\in N\Zd} V_N^l(n-n') \alpha_{n'}^{-1/q} \|\Psi_{n'}^N Ac\|_{q,w}\nonumber\\
 & & +
\sum_{n'\in N\Zd} V_N^{l_0+1}(n-n') \alpha_{n'}^{-1/q} \|\Psi_{n'}^N c\|_{q,w}, \quad \ l_0\ge 1.
 \end{eqnarray}
 Set
\begin{equation} \label{sequencespacesequivalence.tm.pf.eq10}
\epsilon_N^l:=\sum_{k\in N\Zd} \sup_{|n|_\infty\ge |k|_\infty} |V_N^l(n)|.
\end{equation}
 Inductively for $l\ge 2$,
\begin{eqnarray*} \label{sequencespacesequivalence.tm.pf.eq11}
\epsilon_N^{l} &\le   & \epsilon_N^{l-1} \sum_{k\in N\Zd} \sup_{|n|_\infty\ge |k|_\infty/2} |V_N(n)|\nonumber\\
& & +\epsilon_N^1 \sum_{k\in N\Zd} \sup_{|n|_\infty\ge |k|_\infty/2} |V_N^{l-1}(n)|\le 5^{d} \epsilon_N^1 \epsilon_N^{l-1},
 \end{eqnarray*}
where we have used \eqref{banachalgebra.pf.eq6} to obtain the last inequality. This shows that
 \begin{equation} \label{sequencespacesequivalence.tm.pf.eq12}
 \epsilon_N^l\le (5^d \epsilon_N^1)^l \quad  {\rm for\ all}\  l\ge 1.
 \end{equation}
Note that
 \begin{eqnarray*} \label{sequencespacesequivalence.tm.pf.eq8}
\epsilon_N^1 & \le  &   2^{4d+2d/q}  3^{3d} (A_q(w))^{2/q}  \Big\{  N^d
 \Big(\sup_{|i'-j'|_\infty> 4N} |a(i',j')| \Big)
\nonumber\\
 & &
+
  2^{d+1}
 \Big(\sum_{k'\in \Zd,\ |k'|_\infty\ge  \sqrt{N}/2}
 \sup_{|i'-j'|_\infty\ge|k'|_\infty} |a(i',j')|\Big) \nonumber\\
 & &
  + N^{-1/2}\|A\|_{{\mathcal B}} \Big\}
 +
2^{2d} (A_q(w))^{1/q}\nonumber\\
& &\quad \times \Big\{\sum_{|k|_\infty> 8N, k\in N\Zd} N^d \Big(\sup_{|i'-j'|_\infty\ge |k|_\infty/2} |a(i',j')| \Big)\Big\}\nonumber\\
& \le & 2^{4d+2d/q}  3^{3d}  (A_q(w))^{2/q}  \Big\{N^{-1/2}\|A\|_{{\mathcal B}}\nonumber\\
 & & \quad
 +
  2^{d+2}
  \Big(\sum_{k'\in \Zd,\ |k'|_\infty\ge  \sqrt{N}/2}
 \sup_{|i'-j'|_\infty\ge|k'|_\infty} |a(i',j')|\Big)
 \Big\}
  \nonumber\\
 & & +
2^{2d+1} (A_q(w))^{1/q} \sum_{ k'\in \Zd, \ |k'|> 7N} \Big(\sup_{|i'-j'|_\infty\ge 7|k'|_\infty/16} |a(i',j')| \Big)\nonumber\\
\qquad & \le & 2^{6d}  3^{3d} (A_q(w))^{2/q}
 N^{-1/2} \|A\|_{{\mathcal B}}+2^{7d+3} 3^{3d} (A_q(w))^{2/q}   \nonumber\\
 & & \quad \times \Big(\sum_{k'\in \Zd,\ |k'|_\infty\ge  \sqrt{N}/2}
 \big(\sup_{|i'-j'|_\infty\ge |k'|_\infty} |a(i', j')|\big)\Big)\nonumber\\
\quad & \to &  0 \quad {\rm as } \ N\to +\infty
  \end{eqnarray*}
by the
 assumption $A\in {\mathcal B}$.
Let  $N$ be the integer chosen sufficiently large so that
 \begin{equation} \label{sequencespacesequivalence.tm.pf.eq13}
\epsilon_N^1 <5^{-d}.
 \end{equation}
 Taking the limit as $l_0\to \infty$ in \eqref{sequencespacesequivalence.tm.pf.eq14}, and using \eqref{sequencespacesequivalence.tm.pf.eq6++},
\eqref{sequencespacesequivalence.tm.pf.eq12} and \eqref{sequencespacesequivalence.tm.pf.eq13}  lead to
\begin{eqnarray}\label{sequencespacesequivalence.tm.pf.eq15}
\alpha_n^{-1/q}\|\Psi_n^N c\|_{q,w} & \le &\alpha_n^{-1/q}
  \|\Psi_n^N Ac\|_{q,w}\nonumber\\
  & & +
\sum_{n'\in N\Zd} \Big(\sum_{l=1}^{\infty} V_N^l(n-n')\Big) \alpha_{n'}^{-1/q} \|\Psi_{n'}^N Ac\|_{q,w}\nonumber\\
& =:& \sum_{n'\in N\Zd} W_N(n-n')   \alpha_{n'}^{-1/q} \|\Psi_{n'}^N Ac\|_{q,w},
\end{eqnarray}
and
\begin{equation} \label{sequencespacesequivalence.tm.pf.eq16}
\sum_{k\in N\Zd} \Big(\sup_{|n|_\infty\ge |k|_\infty} |W_N(n)|\Big)<\infty.
\end{equation}

Given any $i\in \Zd$, let $n(i)$ be the unique integer  in $N\Zd$
 with $i\in n(i)+\{0, \ldots, N-1\}^{d}$.
Then
\begin{equation} \label{sequencespacesequivalence.tm.pf.eq17}
\alpha_{n(i)}\le \sum_{|i'-i|_\infty<3N} w(i')\le (6N)^{dq} A_q(w) w(i)
\end{equation}
by \eqref{discreteaqweight.eq4}. This together with
\eqref{sequencespacesequivalence.tm.pf.eq15}   implies that
for any $c\in \ell^\infty$,
\begin{eqnarray}\label{sequencespacesequivalence.tm.pf.eq18}
|c(i)| & \le & (6N)^{d} (A_q(w))^{1/q} \alpha_{n(i)}^{-1/q}\|\Psi_{n(i)}^N c\|_{q,w}\nonumber\\
& \le & (6N)^{d}  (A_q(w))^{1/q} \sum_{n'\in N\Zd}
W_N(n(i)-n')\nonumber\\
 & & \times \Big(\sum_{j\in \Zd}  h\big((j-n')/N) |(Ac)(j)|\Big)\nonumber\\
& \le & (6N)^{d} (A_q(w))^{1/q} \nonumber\\
& & \Big\{\sum_{j\in \Zd}
\Big(\sum_{\epsilon\in \{-4, \cdots, 4\}^d}
W_N(n(i-j)+\epsilon N)\Big) |(Ac)(j)|\Big\}\nonumber\\
& =:& \sum_{j\in \Zd} g(i-j) |(Ac)(j)|.\end{eqnarray}
Then the  sequence $\{g(i)\}_{i\in \Zd}$  just defined
satisfies all requirements in Lemma \ref{sequencespacesequivalence.lem}
by \eqref{sequencespacesequivalence.tm.pf.eq16}  and \eqref{sequencespacesequivalence.tm.pf.eq18}.
\end{proof}

\bigskip

Now we proceed to prove Theorem \ref{sequencespacesequivalence.thm}.

\begin{proof} [Proof of Theorem \ref{sequencespacesequivalence.thm}]  By Theorem \ref{banachalgebra.tm}, it suffices to prove the conclusion for any infinite matrix $A\in {\mathcal B}$.

By Theorem \ref{banachalgebra.tm},
 \begin{equation}\label{sequencespacesequivalence.tm.pf.eq20}
\|Ac\|_{q',w'}\le 2^{2d}3^{d/q'}  (A_{q'}(w'))^{1/q'}\|A\|_{\mathcal B}
\|c\|_{q', w'}\quad {\rm for \ all} \ c\in  \ell^{q'}_{w'}. \end{equation}

Let $\{g(i)\}_{i\in \Zd}$ be the sequence in Lemma \ref{sequencespacesequivalence.lem}, and set
\begin{equation}\label{sequencespacesequivalence.tm.pf.eq19}
A_0:=\sum_{k\in \Zd} \Big(\sup_{|i|_\infty \ge |k|_\infty} g(i)\Big)<\infty.
\end{equation}
Then
\begin{eqnarray} \label{sequencespacesequivalence.tm.pf.eq21}
\|c\|_{q', w'} & \le &
\Big\|\Big(\sum_{j\in \Zd} g(i-j) |(Ac)(j)|\Big)_{i\in \Zd}\Big\|_{q',w'}\nonumber\\
& \le &
2^{2d}3^{d/q'} A_0 (A_{q'}(w'))^{1/q'}
\|Ac\|_{q', w'}\quad {\rm for \ all} \ c\in \ell^\infty \cap \ell^{q'}_{w'},
\end{eqnarray}
where the first inequality  follows from \eqref{sequencespacesequivalence.lem.eq2}
and the second inequality holds by Theorem \ref{banachalgebra.tm}.

Combining
\eqref{sequencespacesequivalence.tm.pf.eq20} and \eqref{sequencespacesequivalence.tm.pf.eq21} proves
 the $\ell^{q'}_{w'}$-stability for the infinite matrix $A\in {\mathcal B}$. \end{proof}

Finally we prove Corollary \ref{leftinverse.cor1}.

\begin{proof}[Proof of Corollary \ref{leftinverse.cor1}]
The necessity  is well known, while the sufficiency follows from Theorem \ref{banachalgebra.tm} and  Corollary \ref{leftinverse.cor2}, whose proof will be given in the next section.
\end{proof}

\section{Inverse-closedness}
\label{wiener.section}

In this section, we   prove Theorem \ref{wienerlemmaforbeurlingalgebra.tm}, Corollaries  \ref{beurling.cor} and \ref{leftinverse.cor2},  and the  following Wiener's lemma for  the subalgebra  ${\mathcal B}_{p,u}$ of ${\mathcal B}(\ell^q_w)$.

\begin{thm}\label{wienerlemmaforbeurlingalgebra.tm2}
Let $1\le  p, q<\infty$,  $w$ be a discrete $A_q$-weight,
 $u:=(u(i,j))_{i,j\in \Zd}$ be a $p$-submultiplicative weight matrix  that satisfies
\eqref{matrix.condition0}, \eqref{matrix.condition1}, \eqref{matrix.condition2} and
\begin{equation}\label{wienerlemmaforbeurlingalgebra.tm2.eq1}
M:=\sup_{i\in \Zd} u(i,i)<\infty, \end{equation}
and   $v:=(v(i,j))_{i,j\in \Zd}$ be a companion weight matrix of the $p$-submultiplicative weight matrix
$u$  that satisfies \eqref{matrix.condition3}. If there exist  $D\in (0,\infty)$ and $\theta\in (0,1)$ such that
\begin{equation}\label{wienerlemmaforbeurlingalgebra.tm2.eq2}
   \inf_{N\ge 1}   \big(A_N+  B_N(p) t \big)\le  D t^\theta \quad {\rm for \ all} \ t\ge 1
  \end{equation}
where
\begin{equation} \label{wienerlemmaforbeurlingalgebra.tm2.eq3} A_N:=
\sum_{|k|_\infty\le N}\sup_{|k|_\infty\le |i'-j'|_\infty\le  N}     v(i', ,j')\end{equation}
and
\begin{equation}\label{wienerlemmaforbeurlingalgebra.tm2.eq4}B_N(p):=\Big\|\Big(\sup_{|i'-j'|_\infty\ge |k|_\infty}    v(i', ,j')  (u(i',j'))^{-1}\Big)_{|k|_\infty\ge N/2}\Big\|_{p/(p-1)}, \end{equation}
then ${\mathcal B}_{p,u}$ is an inverse-closed
subalgebra of ${\mathcal B}(\ell^q_w)$.
\end{thm}

One may verify that the  weight matrices $((1+|i-j|)^\alpha)_{i, j\in \Zd}$ with
$ \alpha> d(1-1/p)$,  and $(\exp(|i-j|^\delta))_{i,j\in \Zd}$ with $\delta\in (0,1)$,
and their companion weight matrices satisfy the conditions on weight matrices
  required in Theorem \ref{wienerlemmaforbeurlingalgebra.tm2}  \cite{suntams07}.
   Hence we have the following corollary of  Theorem \ref{wienerlemmaforbeurlingalgebra.tm2}.

\begin{cor}\label{wienerlemmaforbeurlingalgebra.tm2.cor}
Let $1\le  p, q<\infty$,  $w$ be a discrete $A_q$-weight, and let
 $u$ be either $((1+|i-j|_\infty)^\alpha)_{i,j\in \Zd}$ with $\alpha> d(1-1/p)$ or
  $(\exp(|i-j|_\infty^\delta))_{i,j\in \Zd}$ with $\delta\in (0,1)$.
Then ${\mathcal B}_{p,u}$ is an inverse-closed
subalgebra of ${\mathcal B}(\ell^q_w)$.
\end{cor}

\subsection{Proof of Theorem \ref{wienerlemmaforbeurlingalgebra.tm}.}\quad
Let $A\in {\mathcal B}$ have an inverse $A^{-1}\in {\mathcal B}(\ell^q_w)$.
Then
$\|c\|_{q,w}\le \|A^{-1}\|_{{\mathcal B}(\ell^q_w)} \|Ac\|_{q,w}$
 for  all $c\in \ell^q_w$,
where $\|\cdot\|_{{\mathcal B}(\ell^q_w)}$ is the operator norm on ${\mathcal B}(\ell^q_w)$.  Therefore $A$ has $\ell^q_w$-stability.  By Lemma \ref{sequencespacesequivalence.lem}, there exists a sequence $\{g(i)\}_{i\in \Zd}$ such that \eqref{sequencespacesequivalence.lem.eq1} and \eqref{sequencespacesequivalence.lem.eq2} hold.

Write $A^{-1}:=(b(i,j))_{i,j\in \Zd}$,  set
$c_j:=(b(i,j))_{i\in \Zd}$,  and for $l_0\ge 1$ define $c_j^{l_0}:=(b_{l_0}(i,j))_{i\in \Zd}, j\in \Zd$,
where $b_{l_0}(i,j):= b(i,j)$ if $|i-j|_\infty\le l_0$ and $0$ otherwise.
Then $c_j^{l_0}\in \ell^\infty\cap\ell^q_w$ and
\begin{equation}\label{wienerlemmaforbeurlingalgebra.tm.newpf.eq2}
\lim_{l_0\to +\infty} \|c_j^{l_0}-c_j\|_{q,w}=0.
\end{equation}
Applying \eqref{sequencespacesequivalence.lem.eq2} to $c_j^{l_0}$ gives
\begin{equation} \label{wienerlemmaforbeurlingalgebra.tm.newpf.eq3}
|b_{l_0}(i,j)|\le \sum_{i'\in \Zd} g(i-i') |(Ac^{l_0}_j)(i')|, \quad \ i\in \Zd.
\end{equation}
By \eqref{sequencespacesequivalence.lem.eq1}, \eqref{wienerlemmaforbeurlingalgebra.tm.newpf.eq2}, and Theorem \ref{banachalgebra.tm},
\begin{eqnarray}\label{wienerlemmaforbeurlingalgebra.tm.newpf.eq4}
 & & \sum_{i'\in \Zd} g(i-i') |(A(c^{l_0}_j-c_j))(i')|\nonumber\\
 & \le &  w(i)^{-1/q}
\Big\|\Big( \sum_{i^{\prime\prime}\in \Zd} g(i'-i^{\prime\prime}) |(A(c^{l_0}_j-c_j))(i^{\prime\prime})|\Big)_{i'\in \Zd} \Big\|_{q,w}\nonumber\\
 & \le &  2^{4d} 3^{2d/q}  w(i)^{-1/q} (A_q(w))^{2/q}
\|A\|_{{\mathcal B}} \nonumber\\
& & \quad \times \Big\|\Big(\sup_{|j'|_\infty\ge |k|_\infty} |g(j')|\Big)_{k\in \Zd}\Big\|_1 \|c^{l_0}_j-c_j\|_{q,w}\nonumber\\
& \to & 0 \quad  {\rm as} \ l_0\to +\infty.
\end{eqnarray}
 Letting $l_0\to +\infty$ in \eqref{wienerlemmaforbeurlingalgebra.tm.newpf.eq3} and applying \eqref{wienerlemmaforbeurlingalgebra.tm.newpf.eq4} gives
\begin{equation}
\label{wienerlemmaforbeurlingalgebra.tm.newpf.eq5}
|b(i,j)|\le g(i-j) \quad  {\rm for \ all} \ i,j\in \Zd.
\end{equation}
Hence the conclusion $A^{-1}\in {\mathcal B}$ follows from
\eqref{sequencespacesequivalence.lem.eq1} and \eqref{wienerlemmaforbeurlingalgebra.tm.newpf.eq5}.
\hfill $\square$

\subsection{Proof of Corollary \ref{beurling.cor}} Write $f(\xi)=\sum_{n\in \ZZ} a(n) e^{-\sqrt{-1}n\xi}$. Then $A:=(a(i-j))_{i,j\in \ZZ}$ belongs to $ {\mathcal B}$ and  has bounded inverse in ${\mathcal B}(\ell^2)$. Moreover, $A^{-1}=(b(i-j))_{i,j\in \ZZ}$ for  the sequence $b:=(b(n))_{n\in \ZZ}$ determined by $1/f(\xi)=\sum_{n\in \ZZ} b(n) e^{-\sqrt{-1}n\xi}$.
By Theorem \ref{wienerlemmaforbeurlingalgebra.tm}, $A^{-1}\in {\mathcal B}$ which in turn proves  the desired conclusion that  $1/f\in A^*(\TT)$.
\hfill $\square$

\subsection{Proof of Corollary \ref{leftinverse.cor2}}\quad
The necessity follows from Theorem \ref{banachalgebra.tm}.
 Now the sufficiency: Let $1\le q< \infty$, $w$ be a discrete $A_q$-weight, and let $A\in {\mathcal B}$ have $\ell^q_w$-stability.  Then
 $A$ has $\ell^2$-stability by Theorem \ref{beurlingstability.tm},  i.e., there exists a positive constant $C$ such that
$$
C^{-1}\|c\|_2\le \|Ac\|\le C \|c\|_2\quad {\rm for \ all}\ c\in \ell^2.
$$
This implies that $A^*A$ has bounded inverse in ${\mathcal B}(\ell^2)$.  On the other hand, $A^*A$ belong to ${\mathcal B}$ by Proposition \ref{banachalgebra.prop} and
Theorem \ref{banachalgebra.tm}. Therefore
\begin{equation}\label{leftinverse.cor2.pf.eq1}
(A^*A)^{-1}\in {\mathcal B}
 \end{equation}
 by Theorem \ref{wienerlemmaforbeurlingalgebra.tm}.
Now  we prove that
$B:=(A^*A)^{-1}A^*$ is the  desired left inverse of the infinite matrix $A$
in ${\mathcal B}$.
The conclusion that  $B\in {\mathcal B}$ follows from \eqref{leftinverse.cor2.pf.eq1},  Proposition \ref{banachalgebra.prop} and Theorem \ref{banachalgebra.tm}.
From the definition of the infinite matrix $B$,  it defines
 a left inverse  in ${\mathcal B}(\ell^2)$,  it belongs to ${\mathcal B}(\ell^q_w)$ by Theorem \ref{banachalgebra.tm} and $B\in {\mathcal B}$, and the set
$\ell^2\cap \ell^q_w$ is dense in $\ell^q_w$.  Therefore the infinite matrix $B$  is
 a left inverse  in ${\mathcal B}(\ell^q_w)$.
\hfill $\square$

\subsection{Proof of  Theorem \ref{wienerlemmaforbeurlingalgebra.tm2}}\quad To prove Theorem \ref{wienerlemmaforbeurlingalgebra.tm2}, we need a technical lemma. Similar results are established in \cite{suncasp05, suntams07} for infinite matrices in  the Gr\"ochenig-Schur class ${\mathcal S}_{p,u}(\Zd, \Zd)$ and the Gohberg-Baskakov-Sj\"ostrand class ${\mathcal C}_{p,u}(\Zd, \Zd)$, see also Remark \ref{banachalgebra.rem2}.

\begin{lem} \label{wienerlemmaell2.lem2}
Let $1\le p\le \infty$. If   the weight matrix $u$ satisfies
\eqref{matrix.condition0}, \eqref{matrix.condition1},
\eqref{matrix.condition2}, \eqref{wienerlemmaforbeurlingalgebra.tm2.eq1} and
\eqref{wienerlemmaforbeurlingalgebra.tm2.eq2}
 for some positive constants  $D\in (0,\infty)$ and $\theta\in (0,1)$, then
\begin{equation}\label{wienerlemmaell2.lem2.eq1}
\|A^2\|_{{\mathcal B}_{p,u}}\le 2^{2+2/p}5^{(d-1)/p} D \|A\|_{{\mathcal B}_{p,u}}^{1+\theta} \|A\|_{{\mathcal B}(\ell^2)}^{1-\theta}\quad {\rm for \ all} \ A\in {\mathcal B}_{p,u}.
\end{equation}
\end{lem}

\begin{proof}
Let  $A:=(a(i,j))_{i,j\in \Zd}\in {\mathcal B}_{p,u}$, and let
 $A_N$ and  $B_N(p)$ be as in \eqref{wienerlemmaforbeurlingalgebra.tm2.eq3} and \eqref{wienerlemmaforbeurlingalgebra.tm2.eq4} respectively.
Recall that
$|a(i,j)|\le \|A\|_{{\mathcal B}(\ell^2)}$  for  all $i,j\in \Zd$.
 Then for $1<p<\infty$,
\begin{eqnarray*}
& &  \sum_{k'\in \Zd}  |a(i,k')| u(i,k')  |a(k',j)| v(k',j)
\nonumber\\
 & \le &  \inf_{N\ge 1} \Big\{ \|A\|_{{\mathcal B}(\ell^2)}
\sum_{|k'-j|_\infty\le N}  |a(i,k')| u(i,k')  v(k',j)\nonumber\\
& & + \sum_{|k'-j|_\infty>N}  |a(i,k')| u(i,k')  |a(k',j)| v(k',j)\Big\}\nonumber\\
& \le &
\inf_{N\ge 1} \Big\{\|A\|_{{\mathcal B}(\ell^2)}
\Big(\sum_{|k^{\prime\prime}-j|_\infty\le N}    v(k^{\prime\prime} ,j)\Big)^{(p-1)/p}
\nonumber\\
& & \qquad \times
\Big(\sum_{|k'-j|_\infty\le N}  (|a(i,k')| u(i,k'))^p   v(k',j) \Big)^{1/p}\nonumber \\
& & \quad +
\Big(\sum_{|k^{\prime\prime}-j|_\infty>N}|a(k^{\prime\prime},j)| v(k^{\prime\prime},j)\Big)^{(p-1)/p}\nonumber\\
& & \qquad \times \Big(\sum_{|k'-j|_\infty>N}  (|a(i,k')| u(i,k'))^p   |a(k',j)| v(k',j)
\Big)^{1/p}\Big\}. 
\end{eqnarray*}
Therefore we obtain
\begin{eqnarray*}
& & \Big\{\sum_{k\in \Zd} \sup_{|i-j|_\infty\ge |k|_\infty} \Big(\sum_{k'\in \Zd}  |a(i,k')| u(i,k')  |a(k',j)| v(k',j)\Big)^p
\Big\}^{1/p}\nonumber\\
& \le &
\inf_{N\ge 1} \Big\{ \|A\|_{{\mathcal B}(\ell^2)}
(A_N)^{(p-1)/p}\nonumber\\
& &\quad \qquad \times
\Big(A_N \sum_{k\in \Zd}   \sup_{|i'-j'|_\infty\ge |k|_\infty/2}
(|a(i',j')| u(i',j'))^p\nonumber\\
& & \qquad\qquad   +\|A\|_{{\mathcal B}_{p,u}}^p
\sum_{k\in \Zd} \sup_{|k|_\infty/2\le |i'-j'|_\infty\le N}
   v(i',j') \Big)^{1/p}\nonumber \\
& & \qquad + \|A\|_{{\mathcal B}_{p,u}}^{(p-1)/p} (B_N(p))^{(p-1)/p}
 \nonumber\\
 & & \quad\qquad \times \Big(\|A\|_{{\mathcal B}_{p,u}} B_N(p)\sum_{k\in \Zd}
  \sup_{|i'-j'|_\infty\ge |k|_\infty/2 }(|a(i',j')| u(i,j'))^p
  \nonumber\\
& & \qquad\qquad  +\|A\|_{{\mathcal B}_{p,u}}^{p+1}
\Big(\sum_{k\in \Zd}  \sup_{|i'-j'|_\infty\ge \max(|k|_\infty/2,  N)} \nonumber\\
& &  \qquad\qquad\quad   \Big(\frac{v(i', ,j')}{u(i'.j')}\Big)^{p/(p-1)}
\Big)^{(p-1)/p}
\Big)^{1/p}\Big\}\nonumber\\
& \le & 2^{1+2/p}5^{(d-1)/p} \|A\|_{{\mathcal B}_{p,u}}  \inf_{N\ge 1} \Big( \|A\|_{{\mathcal B}(\ell^2)}
A_N
+ B_N(p) \|A\|_{{\mathcal B}_{p,u}}\Big)\nonumber\\
  &\le & 2^{1+2/p}5^{(d-1)/p} D \|A\|_{{\mathcal B}_{p,u}}^{1+\theta}  \|A\|_{{\mathcal B}(\ell^2)}^{1-\theta}.
\end{eqnarray*}
Similarly, we have
\begin{eqnarray*}
& & \Big\{\sum_{k\in \Zd} \sup_{|i-j|_\infty\ge |k|_\infty} \Big(\sum_{k'\in \Zd}  |a(i,k')| v(i,k')  |a(k',j)| u(k',j)\Big)^p
\Big\}^{1/p}\\
  &\le & 2^{1+2/p}5^{(d-1)/p} D \|A\|_{{\mathcal B}_{p,u}}^{1+\theta}  \|A\|_{{\mathcal B}(\ell^2)}^{1-\theta}.
\end{eqnarray*}
Combining the above two estimates and
applying \eqref{banachalgebra.pf.eq1}  with $B=A$, we then get
the desire conclusion \eqref{wienerlemmaell2.lem2.eq1} for $1<p<\infty$.

The conclusion \eqref{wienerlemmaell2.lem2.eq1}
 for $p=1$ and for $p=\infty$ can be established  similarly. We omit the details here.
 \end{proof}

Having the above technical lemma, we can combine the arguments in \cite{suncasp05, suntams07} and  Wiener's lemma for ${\mathcal B}$ to prove Theorem \ref{wienerlemmaforbeurlingalgebra.tm2}.

\begin{proof}[Proof of Theorem \ref{wienerlemmaforbeurlingalgebra.tm2}]
Let $A\in {\mathcal B}_{p,u}$ and $A^{-1}\in {\mathcal B}(\ell^p_w)$.
Then
$A^{-1}\in {\mathcal B}\subset {\mathcal B}(\ell^2)$ by Theorems \ref{wienerlemmaforbeurlingalgebra.tm}
and \ref{banachalgebra.tm}.
This implies that
$ C_1I\le A^*A\le C_2 I$
 for some positive constants $C_1$ and $C_2$, where
$A^*$ is the conjugate transpose of the matrix $A$ and
$I$ is the identity matrix.  Now set
\begin{equation}\label{wienerlemmaforbeurlingalgebra.tm2.pf.eq1}
B:=I-\frac{2}{C_1+C_2} A^*A.
\end{equation}
Then
\begin{equation}\label{wienerlemmaforbeurlingalgebra.tm2.pf.eq2}
\| B\|_{{\mathcal B}(\ell^2)}\le  \frac{C_2-C_1}{C_2+C_1}:=r_0<1.
\end{equation}
On the other hand, $A^*A\in {\mathcal B}_{p,u}$
by Proposition \ref{banachalgebra.prop} and Theorem \ref{banachalgebra.tm}, and
$I\in {\mathcal B}_{p,u}$ by
\eqref{wienerlemmaforbeurlingalgebra.tm2.eq1}. This shows that
\begin{equation}\label{wienerlemmaforbeurlingalgebra.tm2.pf.eq3}
\|B\|_{{\mathcal B}_{p,u}}<\infty.
\end{equation}
Given any integer $n\ge 1$, write  $n=\sum_{l=0}^{l_0} \epsilon_l 2^l$ with $\epsilon_l\in \{0,1\}$. Applying Theorem \ref{banachalgebra.tm} and
Lemma \ref{wienerlemmaell2.lem2} iteratively gives
\begin{eqnarray}
  \|B^n\|_{{\mathcal B}_{p,u}}
&  \le &
(C\|B\|_{{\mathcal B}_{p,u}})^{\epsilon_0}
\|B^{n-\epsilon_0}\|_{{\mathcal B}_{p,u}}\nonumber \\
& \le  &C (C\|B\|_{{\mathcal B}_{p,u}})^{\epsilon_0}
 (\|B\|_{{\mathcal B}(\ell^2)})^{(1-\theta)\sum_{l=0}^{l_0-1}\epsilon_{l+1}2^l}
\nonumber\\
& & \quad \times
(\|B^{\sum_{l=0}^{l_0-1}\epsilon_{l+1}2^l}\|_{{\mathcal B}_{p,u}})^{(1+\theta)}\nonumber \\
& \le & \cdots \le
C^{l_0}(C\|B\|_{{\mathcal B}_{p,u}})^{\sum_{l=0}^{l_0} \epsilon_l(1+\theta)^l}
(\|B\|_{{\mathcal B}(\ell^2)})^{\sum_{l=0}^{l_0} \epsilon_l
(2^l-(1+\theta)^l)}\nonumber\\
 &\le &  C^{\log_2 n}
 \big(Cr_0^{-1}\|B\|_{{\mathcal B}_{p,u}}\big)^{n^{\log_2(1+\theta)}} r_0^n,
\end{eqnarray}
where $C=\max(2^{2+2/p}5^{(d-1)/p}D, 2^{1+2/p} 5^{(d-1)/p} M_p(u))$.
This together with
\eqref{wienerlemmaforbeurlingalgebra.tm2.pf.eq2} and
\eqref{wienerlemmaforbeurlingalgebra.tm2.pf.eq3} shows that
\begin{eqnarray}
  \|A^{-1}\|_{{\mathcal B}_{p,u}} & = & \| (A^*A)^{-1}A^*\|_{{\mathcal B}_{p,u}}\nonumber\\
 &
=&\frac{C_1+C_2}{2} \Big\|A^*+\Big(\sum_{n=1}^\infty B^n\Big)A^*\Big\|_{{\mathcal B}_{p,u}}\nonumber\\
& \le & \frac{C_1+C_2}{2} \Big \{
\|A^*\|_{{\mathcal B}_{p,u}}
+ C \|A^*\|_{{\mathcal B}_{p,u}}\nonumber\\
& & \times\Big(
\sum_{n=1}^\infty C^{\log_2 n}
 \big(Cr_0^{-1}\|B\|_{{\mathcal B}_{p,u}}\big)^{n^{\log_2(1+\theta)}} r_0^n\Big)\Big\}<\infty.
\end{eqnarray}
Hence the conclusion $A^{-1}\in {\mathcal B}_{p,u}$ is proved.
\end{proof}

\bigskip

{\bf Acknowledgements}\quad The author would like to thank Professors Akram Aldroubi,
Karlheinz Gr\"ochenig, Ilya Krishtal, Zuhair Nashed,   Chang Eon Shin,
Wai-Shing Tang  for  their  various helps in the preparation of this manuscript.

\end{document}